\documentclass[12pt]{article}
\usepackage[utf8]{inputenc}
\usepackage{hyperref}
\usepackage{amsmath,amssymb}
\usepackage{amsthm}
\usepackage{amsfonts,tikz}
\usepackage[margin=1in]{geometry}
\newcommand{\bk}{\backslash}
\newcommand{\lm}{\lambda}
\newcommand{\tm}{\widetilde{m}}

\newtheorem{lemma}[]{Lemma}
\newtheorem{cor}[lemma]{Corollary}
\newtheorem{theorem}[lemma]{Theorem}

\theoremstyle{definition}
\newtheorem{remark}[lemma]{Remark}
\tikzstyle{every node}=[circle, draw, fill=black,
                        inner sep=0pt, minimum width=6pt,font=\small]

\title{A Vertex-Weighted Tutte Symmetric Function, and Constructing Graphs with Equal Chromatic Symmetric Function}
\author{Jos\'e Aliste-Prieto\footnote{Departamento de Matemáticas, Universidad Andres Bello, Santiago.\newline Email: jose.aliste@unab.cl.\newline
Supported by CONICYT FONDECYT REGULAR 1160975.},
Logan Crew\footnote{Department of Mathematics, University of Pennsylvania, Philadelphia, PA, 19104.\newline Current address: Department of Combinatorics \& Optimization, University of Waterloo, Waterloo, ON, N2L 3E9.\newline Email: lcrew@uwaterloo.ca.}, Sophie Spirkl\footnote{Department of Mathematics, Princeton University, Princeton, NJ, 08544. \newline Current address: Department of Combinatorics \& Optimization, University of Waterloo, Waterloo, ON, N2L 3E9.\newline  Email:  sspirkl@uwaterloo.ca. \newline This material is based upon work supported by the National Science Foundation under Award No. DMS-1802201. \newline
We acknowledge the support of the Natural Sciences and Engineering Research Council of Canada (NSERC), [funding
reference number RGPIN-2020-03912]. \newline Cette recherche a \'et\'e financ\'ee par le Conseil de recherches en sciences naturelles et en g\'enie du Canada (CRSNG),
[num\'ero de r\'ef\'erence RGPIN-2020-03912]. },
Jos\'e Zamora\footnote{Departamento de Matemáticas, Universidad Andres Bello, Santiago.\newline Email: josezamora@unab.cl.
\newline
Supported by CONICYT FONDECYT REGULAR 1160975.}
}

\date{\today}

\begin{document}

\maketitle

\begin{abstract}

This paper has two main parts. First, we consider the Tutte symmetric function $XB$, a generalization of the chromatic symmetric function. We introduce a vertex-weighted version of $XB$, show that this function admits a deletion-contraction relation, and show that it is equivalent to a number of other vertex-weighted graph functions, namely the $W$-polynomial, the polychromate, and the weighted $(r,q)$-chromatic function. We also demonstrate that the vertex-weighted $XB$ admits spanning-tree and spanning-forest expansions generalizing those of the Tutte polynomial, and show that from this we may also derive a spanning-tree formula for the chromatic symmetric function.

Second, we give several methods for constructing nonisomorphic graphs with equal chromatic and Tutte symmetric functions, and use them to provide specific examples. In particular, we show that there are pairs of unweighted graphs of arbitrarily high girth with equal Tutte symmetric function, and arbitrarily large vertex-weighted trees with equal Tutte symmetric function.
    
\end{abstract}

\textit{\footnotesize Note: this paper was originally announced in \cite{delcon} with the working title ``Using Deletion-Contraction to Construct Graphs with Equal
Chromatic Symmetric Function".}

\section{Introduction}

The chromatic symmetric function $X_G$ of a graph $G$, introduced by Stanley in the 1990s \cite{stanley}, is an extension of the chromatic polynomial that (among other things) counts for each integer partition $\lm = (\lm_1,\dots,\lm_k)$ the number of partitions of $V(G)$ into stable sets of sizes $\lm_1,\dots,\lm_k$.  This function has seen a recent resurgence of interest, including research focusing on the expansion of $X_G$ in the bases of elementary symmetric functions \cite{huh, epos, dahl, foley} and Schur functions \cite{pau, paw, wang}, and the conjecture that $X_G$ distinguishes nonisomorphic trees \cite{trees, heil}.  Other results have extended the definition of $X_G$ in various ways to include representation theoretic and graph theoretic considerations, including chromatic quasisymmetric functions \cite{per, ell, wachs} and chromatic symmetric functions in noncommuting variables \cite{dahl2, noncomm}.  

In \cite{delcon}, the second and third authors extended the chromatic symmetric function to graphs $G$ equipped with vertex weights in the form of a function $w: V(G) \rightarrow \mathbb{N}$.  The extended function $X_{(G,w)}$ satisfies a natural deletion-contraction relation, which can be used to extend identities of $X_G$ to this broader class of graphs, and prove new results.

In this paper, we continue the work of \cite{delcon} in multiple ways.  Following an exposition in Section 2 of necessary background on graphs and symmetric functions, in Section 3 we extend the function $X_{(G,w)}$ to include Stanley's Tutte symmetric function \cite{stanley2}, which is a natural extension of the Tutte polynomial.  We show that the resulting function $XB_{(G,w)}(t,x_1,x_2,\dots)$ of $t$ and variables $x_1,x_2,\dots$ satisfies an edge deletion-contraction relation generalizing that of $X_{(G,w)}$.  

In Section 4 we show that $XB_{(G,w)}$ is a specialization of the $V$-polynomial of Ellis-Monaghan and Moffatt \cite{vpoly}, and is thus closely related to many other graph functions.  For example, we demonstrate that up to a change of variables the vertex-weighted version of $XB$ is equivalent to the $W$-polynomial of Noble and Welsh \cite{noble} by showing that the two functions satisfy the same base cases and recurrence relations, providing a strengthening Noble and Welsh's analogous result on the equivalence of unweighted $XB$ and the unweighted $W$-polynomial (or $U$-polynomial).  We similarly show that the vertex-weighted $XB$ is equivalent to a natural vertex-weighted extension of the polychromate of Brylawski \cite{brylawski} and the weighted $(r,q)$-chromatic function of Klazar, Loebl, and Moffatt \cite{potts}, generalizing proofs of equivalence in unweighted graphs made in the aforementioned references as well as the work of Merino and Noble \cite{merino} and Sarmiento \cite{sarm}.

In Section 5 we use the relationship between $XB$ and the $V$-polynomial to derive spanning-tree and spanning-forest expansions for $XB$.  We show that the spanning-tree expansion specializes to a well-known analogous expansion of the Tutte polynomial, and is a natural improvement of the $p$-basis expansion formula for $XB$ originally introduced in \cite{chowthesis}. From this formula we likewise derive an expansion of the chromatic symmetric function in terms of spanning trees with no external activity, and show that this expansion refines its classical $p$-basis expansion \cite{stanley}.

In Section 6, we use deletion-contraction relations to provide multiple original methods for constructing pairs of nonisomorphic graphs with equal $X_G$ and/or $XB_G$. In particular, we demonstrate graph pairs with arbitrarily high girth whose Tutte symmetric functions agree. 

In Section 7, we further use the equivalence between $XB_G$ and the $W$-polynomial of $G$ to find additional families of vertex-weighted graphs with the same $XB$, and in particular we show how to construct arbitrarily large vertex-weighted paths with equal $XB$ (similar results are found in the independent work \cite{aliniaeifard2020extended} by Aliniaeifard, Wang, and van Willigenburg).  

We conclude in Section 8 with further directions and conjectures.  We note how the examples in Sections 6 and 7 suggest new lines of research related to open problems regarding the chromatic symmetric function, particularly the conjecture that the chromatic symmetric function distinguishes nonisomorphic trees.

\section{Background}

\subsection{Fundamentals of Symmetric Functions and Graphs}

 An \emph{integer partition} (or just \emph{partition}) is a tuple $\lm = (\lm_1,\dots,\lm_k)$ of positive integers such that $\lm_1 \geq \dots \geq \lm_k$.  The integers $\lm_i$ are the \emph{parts} of $\lm$.  If $\sum_{i=1}^k \lm_i = n$, we say that $\lm$ is a partition of $n$, and we write $\lm \vdash n$, or $|\lm| = n$.  The number of parts $k$ is the \emph{length} of $\lm$, and is denoted by $l(\lm)$.  The number of parts equal to $i$ in $\lm$ is given by $r_i(\lm)$.

  A function $f(x_1,x_2,\dots) \in \mathbb{R}[[x_1,x_2,\dots]]$ is \emph{symmetric} if $f(x_1,x_2,\dots) = f(x_{\sigma(1)},x_{\sigma(2)},\dots)$ for every permutation $\sigma$ of the positive integers $\mathbb{N}$.  The \emph{algebra of symmetric functions} $\Lambda$ is the subalgebra of $\mathbb{R}[[x_1,x_2,\dots]]$ consisting of those symmetric functions $f$ that are of bounded degree (that is, there exists a positive integer $n$ such that every monomial of $f$ has degree $\leq n$).  Furthermore, $\Lambda$ is a graded algebra, with natural grading
  $$
  \Lambda = \bigoplus_{d=0}^{\infty} \Lambda^d
  $$
  where $\Lambda^d$ consists of symmetric functions that are homogeneous of degree $d$ \cite{mac,stanleybook}.

  Each $\Lambda^d$ is a finite-dimensional vector space over $\mathbb{R}$, with dimension equal to the number of partitions of $d$ (and thus, $\Lambda$ is an infinite-dimensional vector space over $\mathbb{R}$).  Some commonly-used bases of $\Lambda$ that are indexed by partitions $\lm = (\lm_1,\dots,\lm_k)$ include:
\begin{itemize}
  \item The monomial symmetric functions $m_{\lm}$, defined as the sum of all distinct monomials of the form $x_{i_1}^{\lm_1} \dots x_{i_k}^{\lm_k}$ with distinct indices $i_1, \dots, i_k$.

  \item The power-sum symmetric functions, defined by the equations
  $$
  p_n = \sum_{k=1}^{\infty} x_k^n, \hspace{0.3cm} p_{\lm} = p_{\lm_1}p_{\lm_2} \dots p_{\lm_k}.
  $$
  \item The elementary symmetric functions, defined by the equations
  $$
  e_n = \sum_{i_1 < \dots < i_n} x_{i_1} \dots x_{i_n}, \hspace{0.3cm} e_{\lm} = e_{\lm_1}e_{\lm_2} \dots e_{\lm_k}.
  $$
\end{itemize}

  We also make use of the \emph{augmented monomial symmetric functions}, defined by 
  $$
  \tm_{\lm} = \left(\prod_{i=1}^{\infty} r_i(\lm)!\right)m_{\lm}.
  $$
  
  Given a symmetric function $f$ and a basis $b$ of $\Lambda$, we say that $f$ is \emph{$b$-positive} if when we write $f$ in the basis $b$, all coefficients are nonnegative.

  We define the \emph{symmetric function involution} $\omega$ by $\omega(p_{\lm}) = $ $(-1)^{|\lm|-l(\lm)}p_{\lm}$.

  A \emph{graph} $G = (V,E)$ consists of a \emph{vertex set} $V$ and an \emph{edge multiset} $E$ where the elements of $E$ are (unordered) pairs of (not necessarily distinct) elements of $V$.  An edge $e \in E$ that contains the same vertex twice is called a \emph{loop}.  If there are two or more edges that each contain the same two vertices, they are called \emph{multi-edges}.  A \emph{simple graph} is a graph $G = (V,E)$ in which $E$ does not contain loops or multi-edges (thus, $E \subseteq \binom{V}{2}$).  If $\{v_1,v_2\}$ is an edge, we will write it as $v_1v_2 = v_2v_1$.  The vertices $v_1$ and $v_2$ are the \emph{endpoints} of the edge $v_1v_2$.  We will use $V(G)$ and $E(G)$ to denote the vertex set and edge multiset of a graph $G$, respectively.
  
  Two graphs $G$ and $H$ are said to be \emph{isomorphic} if there exists a bijective map $f: V(G) \rightarrow V(H)$ such that for all $v_1,v_2 \in V(G)$ (not necessarily distinct), the number of edges $v_1v_2$ in $E(G)$ is the same as the number of edges $f(v_1)f(v_2)$ in $E(H)$.

  The \emph{complement} of a simple graph $G = (V,E)$ is denoted $\overline{G}$, and is defined as $\overline{G} = (V, \binom{V}{2} \bk E)$, so in $\overline{G}$ every edge of $G$ is replaced by a nonedge, and every nonedge is replaced by an edge.

  A \emph{subgraph} of a graph $G$ is a graph $G' = (V',E')$ where $V' \subseteq V$ and $E' \subseteq E|_{V'}$, where $E|_{V'}$ is the set of edges with both endpoints in $V'$.  An \emph{induced subgraph} of $G$ is a graph $G' = (V',E|_{V'})$ with $V' \subseteq V$.  The induced subgraph of $G$ using vertex set $V'$ will be denoted $G|_{V'}$.  A \emph{stable set} of $G$ is a subset $V' \subseteq V$ such that $E|_{V'} = \emptyset$.  A \emph{clique} of $G$ is a subset $V' \subseteq V$ such that for every pair of distinct vertices $v_1$ and $v_2$ of $V'$, $v_1v_2 \in E(G)$.

  A \emph{path} in a graph $G$ is a nonempty sequence of edges $v_1v_2$, $v_2v_3$, \dots, $v_{k-1}v_k$ such that $v_i \neq v_j$ for all $i \neq j$.  The vertices $v_1$ and $v_k$ are the \emph{endpoints} of the path.  A \emph{cycle} in a graph is a nonempty sequence of distinct edges $v_1v_2$, $v_2v_3$, \dots, $v_kv_1$ such that $v_i \neq v_j$ for all $i \neq j$.  Note that in a simple graph every cycle must have at least $3$ edges, although in a nonsimple graph there may be cycles of size $1$ (a loop) or $2$ (multi-edges).  

  A graph $G$ is \emph{connected} if for every pair of vertices $v_1$ and $v_2$ of $G$ there is a path in $G$ with $v_1$ and $v_2$ as its endpoints.  The \emph{connected components} of $G$ are the maximal induced subgraphs of $G$ which are connected.  The number of connected components of $G$ will be denoted by $c(G)$.
  
  The \emph{complete graph} $K_n$ on $n$ vertices is the unique simple graph having all possible edges, that is, $E(K_n) = \binom{V}{2}$ where $V = V(K_n)$.

  Given a graph $G$, there are two commonly used operations that produce new graphs.  One is \emph{deletion}: given an edge $e \in E(G)$, the graph of $G$ \emph{with} $e$ \emph{deleted} is the graph $G' = (V(G), E(G) \bk \{e\})$, and is denoted $G \bk e$.  Likewise, if $S$ is a multiset of edges, we use $G \bk S$ to denote the graph $(V(G),E(G) \bk S)$.

  The other operation is the \emph{contraction of an edge} $e = v_1v_2$, denoted $G / e$.  If $v_1 = v_2$ ($e$ is a loop), we define $G / e = G \bk e$.  Otherwise, we create a new vertex $v^*$, and define $G / e$ as the graph $G'$ with $V(G') = (V(G) \bk \{v_1,v_2\}) \cup v^*$, and $E(G') = (E(G) \bk E(v_1, v_2)) \cup E(v^*)$, where $E(v_1,v_2)$ is the set of edges with at least one of $v_1$ or $v_2$ as an endpoint, and $E(v^*)$ consists of each edge in $E(v_1,v_2) \bk e$ with the endpoint $v_1$ and/or $v_2$ replaced with the new vertex $v^*$.  Note that this is an operation on a (possibly nonsimple) graph that identifies two vertices while keeping and/or creating multi-edges and loops.

  There is also a different version of edge contraction that is defined only on simple graphs.  In the case that $G$ is a simple graph, we define the \emph{simple contraction} $G \nmid e$ to be the same as $G / e$ except that after performing the contraction operation, we delete any loops and all but a single copy of each multi-edge so that the result is again a simple graph.  

  Let $G = (V(G),E(G))$ be a (not necessarily simple) graph.  A map $\kappa: V(G) \rightarrow \mathbb{N}$ is called a \emph{coloring} of $G$.  This coloring is called \emph{proper} if $\kappa(v_1) \neq \kappa(v_2)$ for all $v_1,v_2$ such that there exists an edge $e = v_1v_2$ in $E(G)$.  The \emph{chromatic symmetric function} $X_G$ of $G$ is defined as
  
  $$X_G(x_1,x_2,\dots) = \sum_{\kappa} \prod_{v \in V(G)} x_{\kappa(v)}$$ 
  where the sum runs over all proper colorings $\kappa$ of $G$.  Note that if $G$ contains a loop then $X_G = 0$, and $X_G$ is unchanged by replacing each multi-edge by a single edge.  
  
\subsection{Vertex-Weighted Graphs and their Colorings}

A \emph{vertex-weighted graph} $(G,w)$ consists of a graph $G$ and a weight function $w: V(G) \rightarrow \mathbb{N}$.  

Given two vertex-weighted graphs $(G_1,w_1)$ and $(G_2,w_2)$, we call a map $f: V(G_1) \rightarrow V(G_2)$ a \emph{w-isomorphism} if $f$ is an isomorphism of $G_1$ with $G_2$, and also for all $v \in V(G_1)$ we have $w_1(v_1) = w_2(f(v_1))$.

Given a vertex-weighted graph $(G,w)$ and a non-loop edge $e = v_1v_2 \in E(G)$ we define its \emph{contraction by e} to be the graph $(G/e,w/e)$, where $w/e$ is the weight function such that $(w/e)(v) = w(v)$ if $v$ is the not the contracted vertex $v^*$, and $(w/e)(v^*) = w(v_1) + w(v_2)$ (if $e$ is a loop, we define the contraction of $(G,w)$ by $e$ to be $(G \bk e, w)$).

In \cite{delcon}, the authors extended $X_G$ to vertex-weighted graphs as
$$
X_{(G,w)} = \sum_{\kappa} \prod_{v \in V(G)} x_{\kappa(v)}^{w(v)}
$$
where again the sum ranges over all proper colorings $\kappa$ of $G$.  In this setting the chromatic symmetric function admits the deletion-contraction relation \cite{delcon}
\begin{equation}\label{eq:delcon}
X_{(G,w)} = X_{(G \bk e, w)} - X_{(G/e,w/e)}
\end{equation}
as well as the version 
\begin{equation}\label{eq:delconsimple}
  X_{(G,w)} = X_{(G \bk e, w)} - X_{(G \nmid e,w/e)}  
\end{equation}
using simple contraction in the case that $G$ is simple.

Note also that if two vertex-weighted graphs are $w$-isomorphic, then they must have the same chromatic symmetric function.  The converse is not true even in the unweighted case \cite{stanley}. 

\section{The Weighted Version of the Tutte Symmetric Function}

  In this section, we extend the definition of the vertex-weighted chromatic symmetric function to include \emph{all} colorings of a graph, not just the proper ones.  To this end, for a given (not necessarily proper) coloring $\kappa$ of $G$, we define
$$
x_{\kappa}(G,w,t) = (1+t)^{c_{\kappa}(G)} \prod_{v \in V(G)} x_{\kappa(v)}^{w(v)}
$$
where $c_{\kappa}(G)$ is the number of edges $v_1v_2 \in E(G)$ such that $\kappa(v_1) = \kappa(v_2)$.
  We then define the \emph{Tutte symmetric function of a vertex-weighted graph}\footnote{The function is also known as the \emph{bad-coloring chromatic symmetric function}, hence the notation $XB$.  We continue using $XB$ as it is more common in the literature and less confusing in this context than the original $X_G(\textbf{x};t)$.} $(G,w)$ to be the following analogue of the Tutte symmetric function introduced by Stanley in \cite{stanley2}:
\begin{equation}\label{eq:XBdef}
  XB_{(G,w)}(t,x_1,x_2, \dots) = \sum_{\kappa} x_{\kappa}(G,w,t)
\end{equation} 
where the sum is over all colorings $\kappa$ of $G$ (not just the proper ones).  This name comes from the fact that the original function admits the Tutte polynomial $T_G(x,y)$ as a specialization via the relation \begin{equation}\label{eq:XBtutte}
XB_G(t,\underbrace{1,1,\dots,1}_{\text{n 1s}} ,0,0,\dots) = n^{c(G)}t^{|V(G)|-c(G)}T_G\left(\frac{t+n}{n},t+1\right)
\end{equation} 
where $c(G)$ is the number of connected components of $G$.

Given a partition $\pi$ of $V(G)$ (into nonempty blocks), let $e(\pi)$ be the number of edges of $G$ whose endpoints lie in the same block of $\pi$, and $\lm(\pi)$ the partition whose parts are the total weights of the blocks of $\pi$.  We may verify the following lemma, an extension of the corresponding result on unweighted graphs:
\begin{lemma}[\cite{jo}]\label{lem:XBmon}

\begin{equation}\label{eq:XBmon}
XB_{(G,w)} = \sum_{\pi \vdash V(G)} (1+t)^{e(\pi)}\tm_{\lm(\pi)}
\end{equation}
\end{lemma}

\begin{proof}
For $\lm = (\lm_1,\dots,\lm_k)$ it suffices to show that the coefficient of $x_1^{\lm_1} \dots x_k^{\lm_k}$ in $XB_{(G,w)}$ is given by 
$$
\sum_{\substack{\pi \vdash V(G) \\ \lm(\pi) = \lm}} (1+t)^{e(G)}\left(\prod_{i=1}^{\infty} r_i(\lm)!\right). 
$$

From the defining equation \eqref{eq:XBdef} it is clear that we may only get the monomial $x_1^{\lm_1} \dots x_k^{\lm_k}$ by choosing a coloring $\kappa$ of $G$ with vertices of total weight $\lm_i$ receiving the color $i$ for each $i$, and then giving it the coefficient $(1+t)^{c_{\kappa}(G)}$.  Each such coloring $\kappa$ corresponds to a partition $\pi \vdash V(G)$ into $k$ blocks where the $i^{th}$ block consists of vertices colored $i$ by $\kappa$, and these receive a coefficient of $(1+t)^{e(\pi)}$ since the monochromatic edges are exactly those that are entirely contained within a block of $\pi$.  

Conversely, each $\pi \vdash V(G)$ contributes (with coefficient $(1+t)^{e(\pi)}$) exactly $\left(\prod_{i=1}^{\infty} r_i(\lm)!\right)$ colorings with monomial $x_1^{\lm_1} \dots x_k^{\lm_k}$: the one where the $i^{th}$ block of $\pi$ gets color $i$, permuting the color choices among blocks of the same total weight, and the conclusion follows.
\end{proof}

We use the convention $0^0 = 1$, so that when $t = -1$ we have
$$
XB_{(G,w)}(-1,x_1,x_2, \dots) = X_{(G,w)}(x_1,x_2, \dots).
$$

On vertex-weighted graphs, $XB_{(G,w)}$ admits the following deletion-contraction relation that generalizes the deletion-contraction relation of \cite{delcon}:

\begin{lemma}\label{lem:XBdelcon}

Let $(G,w)$ be a vertex-weighted graph.  For all $e \in E(G)$,
\begin{equation}\label{eq:XBdelcon}
XB_{(G,w)} = XB_{(G \bk e, w)} + tXB_{(G / e, w / e)}.
\end{equation}

\end{lemma}

\begin{proof}

First, note that when $t = -1$, the deletion-contraction relation \eqref{eq:XBdelcon} reduces to \eqref{eq:delcon}, so we may assume $t \neq -1$.  Furthermore, the case when $e$ is a loop follows immediately from the definition of $XB$, so we may assume that $e$ is not a loop.  

Let $v_1$ and $v_2$ be the endpoints of $e$.  We start with the right-hand side of \eqref{eq:XBdelcon} and expand using the definition \eqref{eq:XBdef} of $XB$:

 $$
 XB_{(G \bk e, w)} + tXB_{(G / e, w / e)} = $$
 $$\left(\sum_{\kappa: V(G \bk e) \rightarrow \mathbb{N}} x_{\kappa}(G \bk e, w,t)\right) + t\left(\sum_{\kappa: V(G / e) \rightarrow \mathbb{N}} x_{\kappa}(G / e, w / e, t)\right).
 $$

 Note that all colorings of $G$ are also colorings of $G \bk e$, and vice versa.  We split the $\kappa$ in the first summand based on $\kappa(v_1)$ and $\kappa(v_2)$.  In those $\kappa$ where $\kappa(v_1) \neq \kappa(v_2)$, we have $c_{\kappa}(G \backslash e) = c_{\kappa}(G)$, so $x_{\kappa}(G \bk e, w , t) = x_{\kappa}(G, w, t)$.  In all $\kappa$ with $\kappa(v_1) = \kappa(v_2)$, we have $c_{\kappa}(G \backslash e) = c_{\kappa}(G)-1$ because of the missing edge $e$, so for these $\kappa$, we have $x_{\kappa}(G \bk e, w , t) = (1+t)^{-1}x_{\kappa}(G, w, t)$.

 For the second summand, note that every coloring $\kappa$ of $G / e$ corresponds naturally to a coloring $\kappa$ of $G$ with $\kappa(v_1) = \kappa(v_2)$, and vice-versa (we will use the same $\kappa$ to denote both of these colorings in a slight abuse of notation).  For these $\kappa$ we will have $c_{\kappa}(G / e) = c_{\kappa}(G) - 1$ since we are missing the contracted edge $e$, and thus for each such $\kappa$  we will have $x_{\kappa}(G / e, w / e,t) = (1+t)^{-1}x_{\kappa}(G,w,t)$.  Putting everything together, we have

 \begin{align*}
 XB_{(G \bk e, w)} + tXB_{(G / e, w / e)} &= \sum_{\kappa: V(G \bk e) \rightarrow \mathbb{N}} x_{\kappa}(G \bk e, w,t) + t\sum_{\kappa: V(G / e) \rightarrow \mathbb{N}} x_{\kappa}(G / e, w / e, t) \\
 &= \sum_{\substack{\kappa: V(G) \rightarrow \mathbb{N} \\ \kappa(v_1) \neq \kappa(v_2)}} x_{\kappa}(G,w,t) + \sum_{\substack{\kappa: V(G) \rightarrow \mathbb{N} \\ \kappa(v_1) = \kappa(v_2)}} (1+t)^{-1}x_{\kappa}(G,w,t)  \\ &\hphantom{word} +t\sum_{\substack{\kappa: V(G) \rightarrow \mathbb{N} \\ \kappa(v_1) = \kappa(v_2)}} (1+t)^{-1}x_{\kappa}(G,w,t) \\
 &= \sum_{\kappa: V(G) \rightarrow \mathbb{N}} x_{\kappa}(G,w,t) \\ &= XB_{(G,w)}
\end{align*}
as desired.
\end{proof}

As a consequence of this relation, we can derive a $p$-basis expansion formula by simply replacing $(-1)$s with $t$s in (\cite{delcon}, Lemma 3) to give the following analogue of the original formula in \cite{stanley2}:

\begin{cor}
\begin{equation}\label{eq:XBpow}
  XB_{(G,w)}(t,x_1,x_2,\dots)  = \sum_{S \subseteq E(G)} t^{|S|}p_{\lm(G,w,S)}.
\end{equation}
where $\lm(G,w,S)$ is the integer partition whose parts are the total weights of the connected components of $((V(G),S),w)$.
\end{cor}

\section{Relating the Tutte Symmetric Function With Other Graph Functions}

Note that the deletion-contraction relation \eqref{eq:XBdelcon} together with $XB_{(G,w)} = p_{(w_1,\dots,w_k)}$ when $(G,w)$ is a graph with no edges and vertices of weights $w_1 \geq \dots \geq w_k$ can be taken as an alternative definition of $XB_{(G,w)}$.  

This formulation is closely related to the more general $V$-polynomial, defined as a function $V(G,J,w,\{x_j: j \in J\}, \{\gamma_e: e \in E(G)\})$ where 
\begin{itemize}
\item $G$ is a graph;
\item $J$ is a torsion-free commutative semigroup (e.g. $(\mathbb{N}, +)$ or $(2^{\mathbb{N}}, \cap)$);
\item $w: V(G) \rightarrow J$ is a vertex-weight function;
\item The function uses a set of commuting indeterminates $x_j$ indexed by elements $j \in J$, and a set of commuting indeterminates $\gamma_e$ indexed by edges $e \in E(G)$.
\end{itemize}

For brevity, in what follows we will often fix $J$ and the variables $x_j$, and consider $V$ as a function $V_{(G,w)}$ on vertex-weighted graphs. Given these inputs, the $V$-polynomial is defined by the following relations \cite{vpoly}:

\begin{itemize}
\item If $(G,w)$ is a graph with no edges and vertices of weights $w_1, \dots, w_k$, we have $V_{(G,w)} = x_{w_1} \dots x_{w_k}$.
\item If $e \in E(G)$ is a loop, $V_{(G,w)} = (\gamma_e+1)V_{(G \bk e, w)}$.
\item If $e \in E(G)$ is not a loop, $V_{(G,w)} = V_{(G \bk e, w)} + \gamma_e V_{(G / e, w / e)}$.
\end{itemize}

It may be shown from these relations that the $V$-polynomial satisfies \cite{vpoly}
\begin{equation}\label{eq:vsum}
    V_{(G,w)} = \sum_{S \subseteq E(G)} \prod_{c \in C_G(S)} x_{|c|} \prod_{e \in S} \gamma_e
\end{equation}
where $C_G(S)$ is the set of connected components of the graph $(V(G),S)$, and for $c \in C_G(S)$, $|c|$ is the sum (using the operation of $J$) of the weights of the vertices in $c$.

Using either the recurrence relations or the expansion \eqref{eq:vsum}, we may verify that $XB_{(G,w)}$ is a special case of the $V$-polynomial in which $J = (\mathbb{N}, +)$, $\gamma_e = t$ for all $e \in E(G)$, and each variable $x_n$ is replaced by the power-sum symmetric function $p_n(x_1,x_2,\dots)$.  That is,  
\begin{equation}\label{eq:vtoxb}
V(G, (\mathbb{N}, +), w, p_1, p_2, \dots; t) = XB_{(G,w)}(t,x_1,x_2,\dots).
\end{equation}

The function $XB_{(G,w)}$ is also closely related to other specializations of the $V$-polynomial.  A notable example is the $W$-polynomial from invariant theory \cite{noble} (and its unweighted version, the $U$-polynomial), which has been studied both in its own right \cite{noble2} and in relation to the chromatic symmetric function \cite{trees,proper}.

This (nonsymmetric) function $W_{(G,w)}(y,x_1,x_2,\dots)$ on vertex-weighted graphs is defined by the following relations:

\begin{itemize}
\item If $(G,w)$ is a graph with no edges and vertices of weights $w_1 \geq \dots \geq w_k$, we have $W_{(G,w)} = x_{w_1} \dots x_{w_k}$.
\item If $e \in E(G)$ is a loop, $W_{(G,w)} = yW_{(G \backslash e, w)}$.
\item If $e \in E(G)$ is not a loop, $W_{(G,w)} = W_{(G \backslash e, w)} + W_{(G / e, w / e)}$.
\end{itemize}

Note that if $J = (\mathbb{N}, +)$ and $\gamma_e = y-1$ for all $e \in E(G)$, then \begin{equation}\label{eq:vtow}
    V_{(G,w)} = (y-1)^{|V(G)|}W_{(G,w)}(y,x_1(y-1)^{-1}, x_2(y-1)^{-1},\dots)
    \end{equation} 
    so the $W$-polynomial may be derived from the $V$-polynomial \cite{vpoly}.

One can prove either by specializing \eqref{eq:vsum} or induction on the number of edges as in \cite{noble} that the $W$-polynomial satisfies
\begin{equation}\label{eq:WPow}
W_{(G,w)}(y,x_1,x_2,\dots) = \sum_{S \subseteq E(G)} x_{c_1}\dots x_{c_k}(y-1)^{|S|+k-|V(G)|}
\end{equation}
where $c_1,\dots,c_k$ are the total weights of the connected components of the vertex-weighted graph $((V(G),S),w)$.

We say that two functions on vertex-weighted graphs $(G,w)$ are \emph{equivalent} if given one, we can entirely recover the other, without knowing the graph $(G,w)$.  We show the following generalization of (\cite{noble}, Theorem 6.2):
\begin{lemma}\label{lem:Wpoly}
The functions $XB_{(G,w)}(t,x_1,x_2, \dots)$ and $W_{(G,w)}(y,x_1,x_2, \dots)$ are equivalent.
\end{lemma}

\begin{proof}
We actually prove a stronger statement, that given $W_{(G,w)}$, we may recover the $p$-basis expansion of $XB_{(G,w)}$ via the substitution
\begin{equation}\label{eq:XBsub}
XB_{(G,w)} = t^{|V(G)|}W_{(G,w)}\left(t+1,\frac{p_1}{t},\frac{p_2}{t},\dots, \frac{p_k}{t}, \dots\right)
\end{equation}
and conversely, given the $p$-basis expansion of $XB_{(G,w)}$, we may recover $W_{(G,w)}$ by dividing by $t^{|V(G)|}$, setting  $t = y-1$, and replacing each $p_k$ with $tx_k$.  This stronger statement may be proven as a simple vertex-weighted generalization of the argument from (\cite{noble}, Theorem 6.2) by showing that these substitutions take \eqref{eq:XBpow} to \eqref{eq:WPow} and vice-versa.

We provide a different proof by showing that this substitution works not just for these equations, but for the base cases and inductive steps of the recursive definitions for $XB_{(G,w)}$ and $W_{(G,w)}$.  In this sense these functions are not only equivalent, but essentially the same up to a change of variables.

The base cases for both functions are vertex-weighted graphs with no edges.  Let $(G,w)$ be a vertex-weighted graph with no edges and vertices of weights $w_1 \geq \dots \geq w_k$.  Then $XB_{(G,w)} = p_{w_1}\dots p_{w_k}$, $W_{(G,w)} = x_{w_1}\dots x_{w_k}$, and we now verify that the substitution works.  Going from $W$ to $XB$ we have:

$$
x_{w_1}\dots x_{w_k} \mapsto t^k\left(\frac{p_{w_1}}{t}\right) \dots \left(\frac{p_{w_k}}{t}\right) = 
p_{w_1} \dots p_{w_k}
$$
and the converse is analogous.

For the inductive step, assume that we have demonstrated that this substitution is valid for graphs with at most $m$ edges for some $m$.  Let $(G,w)$ be a vertex-weighted graph with $m+1$ edges and let $e$ be an edge of $G$.  Starting with the $W$-polynomial and using deletion-contraction we have two cases.  First, if $e$ is a loop, then $W_{(G,w)} = yW_{(G \bk e, w)}$.  Then applying our substitution we may derive $(t+1)XB_{(G \bk e, w)} = XB_{(G, w)}$, and the converse is analogous.

If $e$ is not a loop, then $W_{(G,w)} = W_{(G \bk e, w)} + W_{(G / e, w / e)}$ (note that $G \bk e$ and $G / e$ have a different number of vertices).  We make the substitution $x_i = \frac{p_i}{t}$, $y = t+1$, and multiply by $t^{|V(G)|}$.  Then by the inductive hypothesis the resulting function is $XB_{(G \bk e, w)} + tXB_{(G / e, w / e)} = XB_{(G,w)}$ as desired, and again the converse process of recovering $W$ from $XB$ is analogous.  
\end{proof}

The function $XB_{(G,w)}$ is also related to the weighted $(r,q)$-chromatic function of $\cite{potts}$.  For a vertex-weighted graph $(G,w)$ with $n$ vertices, this function is defined as 
$$M_{(G,w)}(r,q) = \sum_{S \subseteq E(G)} (-1)^{|S|} \prod_{c \in C(S)} \sum_{i=0}^{n-1}r^{w(c)q^i}$$
where $C(S)$ is the set of connected components of $(V(G),S)$, and $w(c)$ is the total weight of the component $c$.

This function has a natural extension with an additional parameter in the form
\begin{equation}\label{eq:B}
B_{(G,w)}(r,q,t) =  \sum_{S \subseteq E(G)} t^{|S|} \prod_{c \in C(S)} \sum_{i=0}^{n-1}r^{w(c)q^i}.
\end{equation}

Note that from \eqref{eq:B} it is clear that $B_{(G,w)}$ (and thus also $M_{(G,w)}$) may be derived from the $V$-polynomial by taking $J = (\mathbb{N}, +)$ and $\gamma_e = t$ for all $e$, and then substituting $x_n = \sum_{i=0}^{n-1}r^{w(c)q^i}$.

Using the arguments from (\cite{potts}, Section 3) and adjusting them to the vertex-weighted case it is easy to show that

\begin{lemma}\label{lem:rqchrom}
$M_{(G,w)}(r,q)$ is equivalent to $X_{(G,w)}(x_1,x_2,\dots)$, and
$B_{(G,w)}(r,q,t)$ is equivalent to $XB_{(G,w)}(t,x_1,x_2,\dots)$.
\end{lemma}

Finally, we mention a closely related graph function that is not a specialization of the $V$-polynomial.  Retaining the notation used for the $\tm$-basis expansion of $XB$ given in \eqref{eq:XBmon}, define the \emph{polychromate} of a vertex-weighted graph $(G,w)$ as
\begin{equation}\label{eq:polychrom}
\nu_{(G,w)}(y,x_1,x_2,\dots) = \sum_{\pi \vdash V(G)} y^{e(\pi)}x_{\lm(\pi)} 
\end{equation}
where here letting $\lm = \lm(\pi)$ we have $x_{\lm(\pi)} = x_{\lm_1} \cdots x_{\lm_{l(\lm)}}$.

This is a vertex-weighted generalization of a function originally introduced by Brylawski \cite{brylawski}.  Unfortunately, even in this vertex-weighted form, $\nu_{(G,w)}$ is not a specialization of the $V$-polynomial, as when $(G,w)$ has no edges and vertices of weights $w_1,\dots,w_k$ we find that $\nu_{(G,w)}$ is equal to the sum of $x_{w_1}\dots x_{w_k}$ and all of the $x_{\lm}$ where $\lm$ is a coarsening of the partition $(w_1,\dots,w_k)$.  Even modifying the weight set $J$ does not give a reasonable fix to this problem.

Nonetheless, we may easily see by comparing \eqref{eq:polychrom} with \eqref{eq:XBmon} that the vertex-weighted polychromate is equivalent to the vertex-weighted $XB$ and thus to the vertex-weighted $(r,q)$-chromatic function and the $W$-polynomial.  This extends previously known results that showed the equivalence of these four functions on unweighted graphs \cite{potts,merino,sarm}.  A more thorough summary of these functions and their properties on unweighted graphs is given in \cite{jo}.

It is reasonable to ask what advantages are introduced by using the vertex-weighted Tutte symmetric function as opposed to any of these equivalent graph functions. For one, the theory of symmetric function bases may be applied to find encoded information that is much more difficult to detect using the other functions. It is already known that the chromatic and Tutte symmetric functions on vertex-weighted graphs encode information, such as an enumeration of ordered pairs of acyclic orientations and certain maps on their sinks \cite{delcon, modular} or intersections of maximal stable sets \cite{multi}, that to the best of the authors' knowledge have not been replicated by these other functions. Furthermore, as the Tutte polynomial is the universal graph polynomial with a deletion-contraction relation \cite{univ}, it is natural to try to extend its properties to vertex-weighted graphs.  In the next section we will derive further expansions of the Tutte symmetric function that naturally generalize classical expansions of the Tutte polynomial, lending solid evidence to the claim that the Tutte symmetric function is in some sense the natural vertex-weighted extension of the Tutte polynomial.

\begin{section}{Spanning Tree and Spanning Forest Expansions for $XB$}

Properties of the $V$-polynomial specialize naturally to properties of $XB$.  In particular, by considering results in \cite{pottstrees} we may derive spanning tree and spanning forest expansions for $XB$ that are natural generalizations of well-known formulas for the Tutte polynomial.

We will need the following definitions: A \emph{spanning forest} of a graph $G$ is an acyclic subgraph that contains all vertices of $G$.  A \emph{spanning tree} of $G$ is a spanning forest of $G$ with the same number of connected components as $G$.  In what follows, we assume that the edges of $G$ have been given some arbitrary total ordering.  Given a fixed spanning tree $T$ of $G$, we say that an edge $f \in T$ is \emph{internally active with respect to T} if it is the smallest edge in the set $\{e \in E(G): (T \bk f) \cup e \text{ is a spanning tree} \}$, and \emph{internally inactive with respect to T} otherwise.  An edge $f \notin T$ is said to be \emph{externally active with respect to T} if $f$ is the smallest edge in the unique cycle of $T \cup f$, and \emph{externally inactive with respect to T} otherwise.  We may also extend the notion of external activity to spanning forests $F$ by defining an edge $f \notin F$ to be externally inactive with respect to $F$ if $F \cup f$ is acyclic, and otherwise applying the same definition as for trees.

\begin{theorem}[\cite{pottstrees}, Theorems 5.1 and 6.2]
Let $(G,w)$ be a vertex-weighted graph with some arbitrary total order on its edges, and let $T(G)$ be the set of spanning trees of $G$.  For any $T \in T(G)$,  let $ii(T), ia(T), ei(T), ea(T)$ denote respectively the number of internally inactive, internally active, externally inactive, and externally active edges of $G$ with respect to $T \in T(G)$.  Also, let $II(T)$ be the set of internally inactive edges of $G$ with respect to $T$.  Then
\begin{equation}\label{eq:tree}
XB_{(G,w)} = \sum_{T \in T(G)} t^{ii(T)}(t+1)^{ea(T)}XB_{(T,w) / II(T)}.
\end{equation}

Additionally, let $F(G)$ be the set of spanning forests of $G$.  For $F \in F(G)$, let $|F|$ denote the number of edges of $F$, let $ea(F)$ denote the number of externally active edges of $F$, and let $\lm(F)$ denote the partition whose parts are the total weights of the connected components of $F$.  Then
\begin{equation}\label{eq:forest}
XB_{(G,w)} = \sum_{F \in F(G)} t^{|F|}(t+1)^{ea(F)}p_{\lm(F)}.
\end{equation}
\end{theorem}

The spanning forest expansion \eqref{eq:forest} was also already known in an equivalent form for the $W$-polynomial \cite{noble}.  Upon taking $t = -1$ (and $0^0 = 1$) this formula reduces to a vertex-weighted generalization of the chromatic symmetric function analogue of Whitney's Broken Circuit Theorem (\cite{stanley}, Theorem 2.9).

On the other hand, by using the substitution formula \eqref{eq:XBtutte} between $XB$ and the Tutte polynomial, we may verify that equation \eqref{eq:tree} is a direct generalization of the well-known formula
$$
T_G(x,y) = \sum_{T \in T(G)} x^{ia(T)}y^{ea(T)}
$$
which provides further strong justification for the choice of $XB$ as the natural symmetric function analogue of $T_G$.  

Furthermore, retaining the total order on the edges in Theorem 5, consider the mapping $M: 2^{|E(G)|} \rightarrow T(G)$ defined as follows: 
\begin{enumerate}
    \item Input $S \subseteq E(G)$ and the total order of $E(G)$.  We will output $T$, the edges of a spanning tree of $G$.  We start with $T = S$.
    \item Inspect each edge of $S$ in order from least to greatest.  When inspecting an edge $e$, if it is part of a cycle in $(V(G),T)$, remove it from $T$.  
    \item Then, inspect each edge of $G \bk S$ in order from least to greatest.  When inspecting an edge $e$, if adding it to $(V(G),T)$ will not create a cycle, add $e$ to $T$.
    \item Output $M(S) = (V(G),T)$.
\end{enumerate}
Conversely, for $T \in T(G)$, let $II(T)$ be the set of internally inactive edges of $T$, and $EI(T)$ the set of externally inactive edges.  Then it is easy to verify that $M^{-1}(T)$ will consist of those $S \subseteq E(G)$ such that $II(T) \subseteq S$ and $EI(T) \cap S = \emptyset$ (and where $S$ can contain any subset of the internally and externally active edges of $T$).  

Thus, starting from the spanning tree formula \eqref{eq:tree}, if for each $T \in T(G)$ we expand $(t+1)^{ea(T)}$ and $XB_{(T,w)/ II(T)}$ using the $p$-basis expansion \eqref{eq:XBpow}, the monomials of the form $t^ip_{\lm}$ will be exactly those monomials of the $p$-basis formula for the whole graph $(G,w)$ coming from $S \in M^{-1}(T)$.  In this manner, the spanning tree expansion \eqref{eq:tree} also represents a way to refine the full $p$-basis expansion of $XB$.

Furthermore, specializing at $t = -1$ yields
\begin{equation}\label{eq:spanX}
X_{(G,w)} = \sum_{\substack{T \in T(G) \\ ea(T) = 0}} (-1)^{ii(T)}X_{(T,w) / II(T)}
\end{equation}
which provides a spanning tree formula for the chromatic symmetric function.  This is interesting in its own right, and will be discussed further in Section 8.

\end{section}

\begin{section}{Constructing Graphs with Equal Chromatic and Tutte Symmetric Functions}

As with any graph function, it is natural to consider the extent to which the chromatic symmetric function distinguishes nonisomorphic graphs.  The original chromatic symmetric function $X_G$ fails to distinguish nonisomorphic graphs with as few as five vertices \cite{stanley}, and Orellana and Scott \cite{ore} used a modular relation on triangles to construct families of infinitely many pairs of graphs with equal chromatic symmetric function\footnote{In fact, a recent result of Penaguiao \cite{raul} shows that in every pair of nonisomorphic graphs with the same chromatic symmetric function, one may be transformed into the other via finitely many applications of the relation from \cite{ore}.}.

The deletion-contraction relation on $X_{(G,w)}$ and $XB_{(G,w)}$ is a powerful and simple tool for finding such constructions.  Throughout this section, we will repeatedly use the following simple but fundamental lemma:

\begin{lemma}\label{lem:eqx}
Let $(G_1,w_1)$ and $(G_2,w_2)$ be vertex-weighted graphs, and let $e_1 \in E(G_1)$ and $e_2 \in E(G_2)$ be edges.
\begin{itemize}
\item If $X_{(G_1 \bk e_1, w_1)} = X_{(G_2 \bk e_2, w_2)}$ and $X_{(G_1 \nmid e_1, w_1 / e_1)} = X_{(G_2 \nmid e_2, w_2 / e_2)}$, then $X_{(G_1,w_1)} = X_{(G_2,w_2)}$.
\item If $XB_{(G_1 \bk e_1, w_1)} = XB_{(G_2 \bk e_2, w_2)}$, and $XB_{(G_1 / e_1, w_1 / e_1)} = XB_{(G_2 / e_2, w_2 / e_2)}$, then $XB_{(G_1,w_1)} = XB_{(G_2,w_2)}$.
\end{itemize}
\end{lemma}

\begin{proof}
These claims follow immediately from equation \eqref{eq:delconsimple} and Lemma \ref{lem:XBdelcon} respectively.
\end{proof}

In particular both parts of Lemma \ref{lem:eqx} hold when their equalities are replaced by $w$-isomorphism of the corresponding graphs.  Note that in the case of $XB$ we cannot replace contraction by simple contraction because unlike with the chromatic symmetric function, graphs that differ only by replacing edges with multi-edges or vice versa still have different $XB$. 

\begin{subsection}{Constructions with Split Graphs}

A \emph{bipartite graph} is a graph that has a proper $2$-coloring, that is, a graph whose vertices may be partitioned into two stable sets.  A \emph{split graph} is a graph that arises from taking a simple bipartite graph $G$ with $V(G)$ partitioned into nonempty stable sets $S_1$ and $S_2$, and switching all of the nonedges in either (but not both) of $G|_{S_1}$ and $G|_{S_2}$ to edges.  Thus, the vertices of a split graph may be partitioned (not necessarily uniquely) into a stable set and a clique.  The class of split graphs can also be characterized by the property that they contain no induced subgraph isomorphic to a five-vertex cycle, a four-vertex cycle, or the complement of a four-vertex cycle $\cite{split}$.

  There is a natural way noted by Loebl and Sereni \cite{loebl} to associate to any possibly non-simple (unweighted) graph a corresponding simple split graph: given a graph $G = (V,E)$, with $V(G) = \{v_1,\dots,v_n\}$ and $E(G) = \{e_1,\dots,e_m\}$.  Then the \emph{split graph sp(G) corresponding to G}  has vertex set $V(sp(G)) = \{t_1,t_2,\dots,t_n,t_{n+1},\dots,t_{n+m}\}$, and edge set $E(sp(G)) = \{t_it_j : 1 \leq i < j \leq n\} \cup \{t_it_{n+j}, t_{i'}t_{n+j} : e_j = v_iv_{i'} \text{ in } G\}$.  In other words, $sp(G)$ is formed by taking the vertices of $G$, making them into a clique, and then adding a ``hat'' corresponding to each edge of $G$.  Using the above notation, we say that vertex $t_{n+j}$ of $sp(G)$ is the \emph{splitting vertex} of the edge $e_j = v_iv_{i'}$ in $G$.  The construction is illustrated in Figure \ref{Fig:spExample}.

\begin{figure}[!htb]
\begin{center}
\begin{tikzpicture}[scale=1.5]

\node[fill=black, circle] at (0, 1)(v1){};
\node[fill=black, circle] at (1, 1)(v2){};
\node[fill=black, circle] at (0, 0)(v3){};
\node[fill=black, circle] at (1, 0)(v4){};

\draw[black, thick] (v1) -- (v2);
\draw[black, thick] (v1) -- (v3);
\draw[black, thick] (v2) -- (v4);

\draw (2, 0.5) coordinate (RA) node[fill=none,draw=none, right] {\LARGE $\bf{\rightarrow} $};
\end{tikzpicture}
\hspace{1cm}
\begin{tikzpicture}[scale=1.5]

\node[fill=black, circle] at (3, 1)(v1){};
\node[fill=black, circle] at (4, 1)(v2){};
\node[fill=black, circle] at (3, 0)(v3){};
\node[fill=black, circle] at (4, 0)(v4){};
\node[fill=black, circle] at (2.3, 0.5)(v13){};
\node[fill=black, circle] at (4.7, 0.5)(v24){};
\node[fill=black, circle] at (3.5, 1.7)(v12){};

\draw[black, thick] (v1) -- (v2);
\draw[black, thick] (v1) -- (v3);
\draw[black, thick] (v1) -- (v4);
\draw[black, thick] (v2) -- (v3);
\draw[black, thick] (v2) -- (v4);
\draw[black, thick] (v3) -- (v4);
\draw[black, thick] (v1) -- (v12);
\draw[black, thick] (v2) -- (v12);
\draw[black, thick] (v1) -- (v13);
\draw[black, thick] (v3) -- (v13);
\draw[black, thick] (v2) -- (v24);
\draw[black, thick] (v4) -- (v24);

\end{tikzpicture}
\end{center}
\caption{An example of the split graph construction.}
\label{Fig:spExample}
\end{figure}
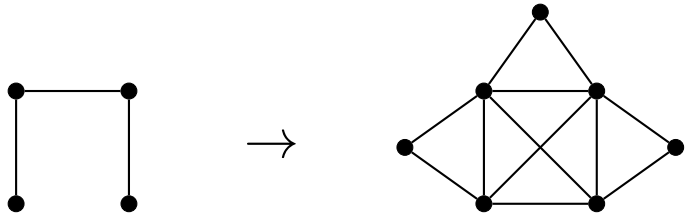

For any two nonisomorphic graphs of more than three vertices the corresponding split graphs are clearly nonisomorphic (since the largest clique of the split graph must correspond to the vertex set of the original graph), so distinguishing split graphs is (up to some processing) equivalent to distinguishing all graphs.  This motivates considering which functions distinguish split graphs; in \cite{loebl} Loebl and Sereni conjecture that the $U$-polynomial (equivalently the Tutte symmetric function $XB$) does.

It is natural to consider whether the chromatic symmetric function itself already distinguishes
split graphs. Unfortunately it does not, and in particular, the following lemma allows for the
construction of infinitely many pairs of split graphs that have equal chromatic symmetric functions.

This construction will make use of graph automorphisms.  An \emph{automorphism} of a graph $G$ is an isomorphism $f$ of $G$ with itself, and likewise a $w$\emph{-automorphism} of a vertex-weighted graph $(G,w)$ is a map $f$ that is a $w$-isomorphism of $(G,w)$ with itself.

Additionally, for $v_1,v_2 \in V(G)$, if $v_1v_2 \notin E(G)$, we use the shorthand $G \cup v_1v_2$ to mean the graph $(V(G), E(G) \cup v_1v_2)$.  For brevity if $v \in V(G)$ we also use $v$ to refer to the corresponding vertex of $sp(G)$.

\begin{lemma}\label{Lem:Split}

  Let $G$ be an unweighted graph.  Suppose $G$ has (not necessarily distinct) vertices $u,u',v,v'$ such that:
  \begin{itemize}
  \item $uv \notin E(G)$ and $u'v' \notin E(G)$.
  \item There is some automorphism of $G$ that maps $u$ to $u'$, and some (possibly different) automorphism of $G$ that maps $v$ to $v'$.
  \end{itemize}
  Then $X_{sp(G \cup uv)} = X_{sp(G \cup u'v')}$.
\end{lemma}

\begin{proof}
  Throughout this proof we will omit the weight function $w$ from $(G,w)$; the vertex weights will always all be 1 unless otherwise specified.
  Let $G$ be as stated.  In $sp(G \cup uv)$, let $x$ be the splitting vertex of $uv$, and likewise in $sp(G \cup u'v')$ let $x'$ be the splitting vertex of $u'v'$.  By applying Lemma \ref{lem:eqx} to edge $ux$ of $sp(G \cup uv)$ and edge $u'x'$ of $sp(G \cup u'v')$ it suffices to show that the graphs $sp(G \cup uv) \bk ux$ and $sp(G \cup u'v') \bk u'x'$ are $w$-isomorphic, and that the graphs $sp(G \cup uv) \nmid ux$ and $sp(G \cup u'v') \nmid u'x'$ are $w$-isomorphic.

  Note that if $f: V(G) \rightarrow V(G)$ is an automorphism of $G$, we may extend it to an automorphism of $sp(G)$ by defining that for $z \in V(sp(G)) \bk V(G)$, if $z$ is the splitting vertex of $ab$, $f(z)$ is the splitting vertex of $f(a)f(b)$.

  Let $G_x$ denote $sp(G \cup uv) \bk ux $, and let $G_{x'}$ denote $sp(G \cup u'v') \bk u'x'$.  Then $V(G_x) \bk \{x\} = V(G_{x'}) \bk \{x'\}$ and $E(G_x) \bk \{vx\} = E(G_{x'}) \bk \{v'x'\}$.  By hypothesis there is an automorphism $f$ of $G$ with $f(v) = v'$, which may be extended to an automorphism of $sp(G)$ as described above.  It is easy to verify that if we extend $f$ once more to a function $f: V(G_x) \rightarrow V(G_{x'})$ by defining $f(x) = x'$, then $f$ is a $w$-isomorphism of the (unweighted) graphs $G_x$ and $G_{x'}$.
  
We now address the graphs with contracted edges. Upon applying simple contraction to the edge $ux \in sp(G \cup uv)$, we let $z$ be the vertex formed by contraction (now with weight $2$), and as we are applying simple contraction. Likewise, when applying contraction to the edge $u'x' \in sp(G \cup u'v')$, we let $z'$ be the vertex formed by contraction (now with weight $2$).

  Let $G_{z}$ denote $sp(G \cup uv) \nmid ux$ and let $G_{z'}$ denote $sp(G \cup u'v') \nmid u'x'$.  Then $V(G_z) \bk \{z\} = V(G_{z'}) \bk \{z'\}$.  By hypothesis there is an automorphism of $G$ taking $u$ to $u'$ that extends to an automorphism of $sp(G)$.  By extending $f$ to a function $f : V(G_z) \rightarrow V(G_{z'})$ with $f(z) = z'$ (instead of $f(u) = u'$), this $f$ is a $w$-isomorphism of $sp(G \cup uv) \nmid ux$ and $sp(G \cup u'v') \nmid u'x'$.
\end{proof}

Thus, when $G \cup uv$ is not isomorphic to $G \cup u'v'$, the graphs $sp(G \cup uv)$ and $sp(G \cup u'v')$ are nonisomorphic split graphs with equal chromatic symmetric functions.  One way to generate such examples easily is by taking an arbitrary noncomplete connected graph $G$, and choosing any nonedge $ab$ in $G$.  Then we construct $2G$ as the disjoint union of graphs $G$ and $G^*$, where $G^*$ is isomorphic to $G$ (that is, $V(2G) = V(G) \sqcup V(G^*)$, and $E(2G) = E(G) \sqcup E(G^*)$).  Let $f: V(G) \rightarrow V(G^*)$ be an isomorphism of $G$ and $G^*$.  In the statement of Lemma \ref{Lem:Split}, let $u = u' = a$ and $v = b$ be vertices of the component $G$, and $v' = f(b)$ a vertex of the component $G^*$.  Then it is simple to verify that these choices for $u,u',v,v'$ satisfy the lemma, and that the two graphs $2G \cup uv$ and $2G \cup u'v'$ are nonisomorphic, since the latter is connected, and the former is not.

We can also use Lemma \ref{Lem:Split} to produce two nonisomorphic graphs, both connected, such that their split graphs have equal chromatic symmetric functions, as shown in Figure \ref{Fig:spConnected}.  Note that in this figure $G \cup uv$ is not isomorphic to $G \cup u'v'$ because, for example, $G \cup uv$ contains a triangle ($K_3$), and $G \cup u'v'$ does not.

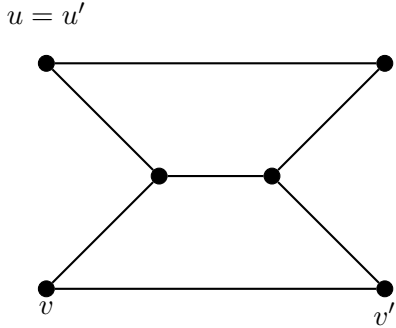
\begin{figure}[!htb]
\begin{center}
\begin{tikzpicture}[scale=1.5]
    
\node[label=above:{$u = u'$}, fill=black, circle] at (0, 2)(v1){};
\node[fill=black, circle] at (3, 2)(v2){};
\node[fill=black, circle] at (1, 1)(v3){};
\node[fill=black, circle] at (2, 1)(v4){};
\node[label=below:{$v$}, fill=black, circle] at (0, 0)(v5){};
\node[label=below:{$v'$}, fill=black, circle] at (3, 0)(v6){};

\draw[black, thick] (v1) -- (v2);
\draw[black, thick] (v1) -- (v3);
\draw[black, thick] (v2) -- (v4);
\draw[black, thick] (v3) -- (v4);
\draw[black, thick] (v3) -- (v5);
\draw[black, thick] (v4) -- (v6);
\draw[black, thick] (v5) -- (v6);

\end{tikzpicture}
\end{center}
\caption{An unweighted connected graph $G$ such that $X_{sp(G \cup uv)} = X_{sp(G \cup u'v')}$.}
\label{Fig:spConnected}
\end{figure}

However, Lemma \ref{Lem:Split} can not generalize directly to $XB$ because $\nmid$ does not admit a simple deletion-contraction relation on $XB$.  If we instead use normal contraction $/$ on the edge $ux$, we get a multi-edge between $u$ and $v$, and likewise for $u'$ and $v'$.  Thus, to generalize Lemma \ref{Lem:Split} we would need a single automorphism of $G$ that takes $u$ to $u'$ and $v$ to $v'$ simultaneously; but then clearly $G \cup uv$ and $G \cup u'v'$ would be isomorphic!
\end{subsection}

\begin{subsection}{Further Constructions of Graphs with Equal $X_G$}

  In much of the recent literature on the chromatic symmetric function, examples of pairs of graphs with equal chromatic symmetric function have been generated using a result of Orellana and Scott.  We reiterate it here and also prove that it extends to vertex-weighted graphs:

\begin{theorem}[\cite{ore}, Theorem 4.2]\label{thm:ore}

  Let $(G,w)$ be a simple, vertex-weighted graph with distinct vertices $v_1,v_2,v_3,v_4$ such that
  \begin{itemize}
  \item $v_1v_2, v_2v_3, v_3v_4 \in E(G)$, and $v_1v_3,v_1v_4,v_2v_4 \notin E(G)$.
  \item There is a $w$-automorphism $f$ of $G \bk v_2v_3$ such that $f(\{v_1,v_3\}) = \{v_2,v_4\}$ and $f(\{v_2,v_4\}) = \{v_1,v_3\}$.
  \end{itemize}
  Then the graphs $G \cup v_1v_3$ and $G \cup v_2v_4$ have equal chromatic symmetric function.

\end{theorem}

\begin{proof}
Let $G_1 = G \cup v_1v_3$ and let $G_2 = G \cup v_2v_4$.  By applying Lemma \ref{lem:eqx} to edges $v_1v_3$ of $G_1$ and $v_2v_4$ of $G_2$, it suffices to show that $G_1 \nmid v_1v_3$ is $w$-isomorphic to $G_2 \nmid v_2v_4$.

The portions of these graphs induced by $v_1,v_2,v_3,v_4$ and their contractions are illustrated in Figure \ref{Fig:Ore}.  It is clear from this figure that the automorphism $f$ given by hypothesis induces a $w$-isomorphism of $G_1 \nmid v_1v_3$ and $G_2 \nmid v_2v_4$, so we are done.

\end{proof}

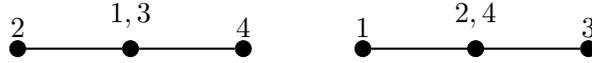
\begin{figure}[!htb]
\begin{center}
\begin{tikzpicture}[scale=1.5]

\node[label=above:{$2$}, fill=black, circle] at (0, 0)(v1){};
\node[label=above:{$1,3$}, fill=black, circle] at (1, 0)(v2){};
\node[label=above:{$4$}, fill=black, circle] at (2, 0)(v3){};

\draw[black, thick] (v1) -- (v2);
\draw[black, thick] (v2) -- (v3);
\end{tikzpicture}
\hspace{1cm}
\begin{tikzpicture}[scale=1.5]

\node[label=above:{$1$}, fill=black, circle] at (0, 0)(v1){};
\node[label=above:{$2,4$}, fill=black, circle] at (1, 0)(v2){};
\node[label=above:{$3$}, fill=black, circle] at (2, 0)(v3){};

\draw[black, thick] (v1) -- (v2);
\draw[black, thick] (v2) -- (v3);

\end{tikzpicture}
\end{center}
\caption{The portions of $G_1 \nmid v_1v_3$ and $G_2 \nmid v_2v_4$ induced by $v_1,v_2,v_3,v_4$.}
\label{Fig:Ore}
\end{figure}

In addition to Lemma \ref{Lem:Split} and Theorem \ref{thm:ore}, we present one more method for constructing graphs with equal chromatic symmetric function.  This method is inspired by the case $u = u'$ of Lemma \ref{Lem:Split}, but can be used in slightly more general contexts and is more akin to Theorem \ref{thm:ore}.

Given a simple graph $G$ and a vertex $v \in V(G)$, we define the \emph{neighborhood of} $v$ to be $N(v) = \{u \in V(G): uv \in E(G)\}$ (note that $v \notin N(v)$).

\begin{lemma}\label{lem:eqX}

  Let $(G,w)$ be a simple vertex-weighted graph, and let $v_1,v_2,v_3$ be distinct vertices of $G$ satisfying

  \begin{itemize}
    \item $v_1v_2 \in E(G)$, and $v_1v_3, v_2v_3 \notin E(G)$.
    \item $N(v_3) \subseteq N(v_1) \cap N(v_2)$.
    \item There is a $w$-automorphism $f$ of $G \bk v_3$ such that $f(v_1) = v_2$ and $f(v_2) = v_1$.
  \end{itemize}

  Then the graphs $G \cup v_1v_3$ and $G \cup v_2v_3$ have equal chromatic symmetric functions.

\end{lemma}

\begin{proof}

  We let $e_1 = v_1v_3$ and $e_2 = v_2v_3$ be nonedges of $G$. By applying Lemma \ref{lem:eqx} to edge $e_1$ of $G \cup e_1$ and to edge $e_2$ of $G \cup e_2$, it suffices to show that $(G \nmid e_1, w / e_1)$ and $(G \nmid e_2, w / e_2)$ are $w$-isomorphic (and from now on, we suppress mention of the weight functions).
  
  In what follows, we let $u_1$ be the contracted vertex (of weight $2$) in $G \nmid e_1$, and we let $u_2$ be the contracted vertex (of weight $2$) in $G \nmid e_2$.  Furthermore, $G \nmid e_1$ contains $v_2$ but not $v_1$, and $G \nmid e_2$ contains $v_1$ but not $v_2$, and otherwise these graphs have the same vertex set, all of weight $1$ except for $u_1$ or $u_2$.
  
  Let $f$ be the $w$-automorphism of $G \bk v_3$ that swaps $v_1$ and $v_2$.  We define the map $g: V(G \nmid e_1) \rightarrow V(G \nmid e_2)$ by $g(v) = f(v)$ if $v \neq u_1,v_2$, $g(u_1) = u_2$, and $g(v_2) = v_1$.  Clearly this $g$ is a $w$-isomorphism if it is an isomorphism.  All edge and nonedge relations between vertices of $G \nmid e_1$ other than $u_1$ and $v_2$ are preserved in $G \nmid e_2$ by $g$ since they were preserved by $f$, so it suffices to look at edges and nonedges involving $u_1$ and $v_2$.

 Let $G_1 = G \nmid e_1$ and $G_2 = G \nmid e_2$.  Using the definition of contraction and the hypotheses we have
 $$
 N_{G_1}(u_1) = N_G(v_1) \cup N_G(v_3) = N_G(v_1),
 $$
 and
 $$
 N_{G_2}(u_2) = N_G(v_2) \cup N_G(v_3) = N_G(v_2) = g(N_G(v_1))
 $$
 so the neighborhood of $u_1$ is mapped to the neighborhood of $u_2$ by $g$.  Additionally,
 $$
 g(N_{G_1}(v_2)) = g(N_G(v_2)) = N_G(v_1) = N_{G_2}(v_1) 
 $$
 so the neighborhood of $v_2$ is mapped to the neighborhood of $v_1$, and this concludes the proof.

\end{proof}

\end{subsection}

\begin{subsection}{Graphs with Equal $XB$}

Relative to $X_G$, only a few examples of nonisomorphic graphs with equal Tutte symmetric functions are known.  An example with a minimum number of vertices and edges is given by Markstrom \cite{ising}\footnote{The example of \cite{ising} is also Example 259 of the authors' list of pairs of graphs with equal chromatic symmetric function \cite{data}.}.  Additionally, Brylawski \cite{brylawski} uses the rotor-like graph given in Figure \ref{fig:brylawksi} to construct a family of graph pairs with arbitrarily high connectivity and equal Tutte symmetric function\footnote{This construction depends in part on the fact that the Tutte symmetric function of a simple graph $G$ uniquely determines that of $G$'s complement.  This fact is not obvious from the definition of $XB$, but is clear from the definition of the equivalent polychromate.}.

We modify Brylawski's result to construct graph pairs with arbitrarily high girth and equal Tutte symmetric function.

\begin{theorem}\label{thm:girthXB}
Let $k > 2$ be a positive integer, and let $(G,w)$ be a (not necessarily simple) vertex-weighted graph with distinct vertices $a,b,$ and $c$ such that there exists a $w$-automorphism $f$ of $(G,w)$ satisfying $f(a) = b, f(b) = c,$ and $f(c) = a$.  Modify the graph $(G,w)$ by replacing every edge with an unweighted path of length $k$ $($that is, for each edge $u_1u_2 = e \in E(G)$, delete $e$, add weight-$1$ vertices $e_1, \dots, e_{k-1}$ to $G$, and add edges $u_1e_1, e_1e_2, \dots, e_{k-1}u_2$ to $G)$.  Then construct the vertex-weighted graphs $(G_1,w)$ and $(G_2,w)$ by retaining the weight function $w$ and setting 
\begin{itemize}
    \item $V(G_1) = V(G_2) = V(G) \cup \{v_1, v_2, \dots, v_{k-1}, x_1, x_2, \dots, x_k\}$ $($where these vertices all have weight $1)$,
    \item $E(G_1) = E(G) \cup \{av_1, v_1v_2, \dots, v_{k-1}b, bx_1, x_1x_2, \dots, x_kc\}$,
    \item $E(G_2) = E(G) \cup \{av_{k-1}, v_{k-1}v_{k-2}, \dots, v_1c, cx_1, x_1x_2, \dots, x_kb\}$.
\end{itemize} 
Then 
$$
XB_{(G_1,w)} = XB_{(G_2,w)}.
$$
\end{theorem}

\begin{proof}

Note that the modified version of $G$ still admits $f$ as a $w$-automorphism with $f(a) = b, f(b) = c,$ and $f(c) = a$ by extending its definition to the newly formed vertices.
We apply Lemma \ref{lem:eqx} to the edges $bx_1$ of $(G_1,w)$ and $cx_1$ of $(G_2,w)$.  We may check that the graph $(G_1 \bk bx_1, w)$ is $w$-isomorphic to $(G_2 \bk cx_1, w)$ by extending $f$ to a map $f: V(G_1) \rightarrow V(G_2)$ and defining $f(v_i) = v_i$ and $f(x_i) = x_i$ for all appropriate $i$, and verifying that $f^2$ is the desired $w$-isomorphism.  Likewise, if we instead define $f(v_i) = x_{k-i+1}$ and $f(x_i) = v_{i-1}$, we may verify that $f$ is a $w$-isomorphism from $(G_1 / bx_1, w / bx_1)$ to $(G_2 / cx_1, w / cx_1)$ upon also letting $f$ map the vertex formed by contraction in $G_1$ to the vertex formed by contraction in $G_2$, and this concludes the proof.  

\end{proof}

Thus, we may use Theorem \ref{thm:girthXB} to construct pairs of (unweighted) graphs of arbitrarily high girth with equal Tutte symmetric function given a single graph $G$ such that the resulting graphs $G_1$ and $G_2$ are always nonisomorphic.  It is straightforward to verify that the graph of Brylawski as given in Figure \ref{fig:brylawksi} is one such example.   This construction is particularly noteworthy since there has not previously been shown even triangle-free graphs with the same chromatic symmetric function!

\begin{figure}[!htb]
\begin{center}
\begin{tikzpicture}[scale=1.5]

  \node[label=left:{a}, fill=black, circle] at (0, 0)(a){};
  \node[label=above:{b}, fill=black, circle] at (2, 3.46)(b){};
  \node[label=right:{c}, fill=black, circle] at (4, 0)(c){};
  \node[fill=black, circle] at (2, 1.15)(cen){};
  \node[fill=black, circle] at (1.33, 0.77)(aa){};
  \node[fill=black, circle] at (2.67, 0.77)(cc){};
  \node[fill=black, circle] at (2, 1.91)(bb){};
  \node[fill=black, circle] at (1.33, 1.53)(ab){};
  \node[fill=black, circle] at (2, 0.39)(ac){};
  \node[fill=black, circle] at (2.67, 1.53)(bc){};

  \draw[black, thick] (a) -- (aa);
  \draw[black, thick] (b) -- (bb);
  \draw[black, thick] (c) -- (cc);
  \draw[black, thick] (cen) -- (aa);
  \draw[black, thick] (cen) -- (bb);
  \draw[black, thick] (cen) -- (cc);
  \draw[black, thick] (ab) -- (bb);
  \draw[black, thick] (bc) -- (cc);
  \draw[black, thick] (ac) -- (aa);
  \draw[black, thick] (ab) -- (a);
  \draw[black, thick] (bc) -- (b);
  \draw[black, thick] (ac) -- (a);
  \draw[black, thick] (ab) -- (b);
  \draw[black, thick] (bc) -- (c);
  \draw[black, thick] (ac) -- (c);
  
\end{tikzpicture}
\end{center}
\caption{A graph $G$ giving rise to nonisomorphic graphs $G_1$ and $G_2$ with the same Tutte symmetric functions and arbitrarily large girth.}
\label{fig:brylawksi}
\end{figure}
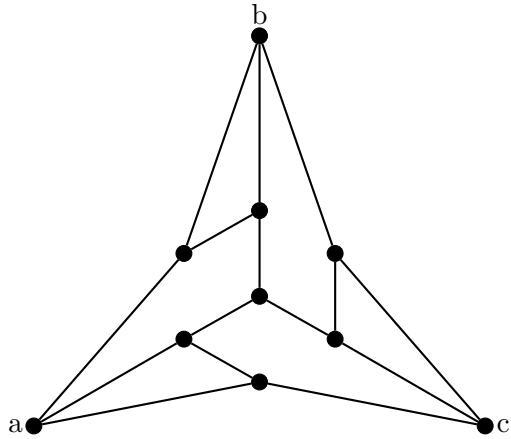

We also give two pairs of small graphs with equal Tutte symmetric functions that do not fit the requirements of Theorem \ref{thm:girthXB}.  In the figures that follow, the numbers next to the vertices are labels rather than weights, so that graphs can be redrawn to illustrate isomorphisms.  Vertex weights from an edge contraction will be denoted by simply listing each original vertex label that corresponds to a vertex formed by edge contraction.

First, we consider the graphs shown in Figure \ref{fig:graphEqual}.

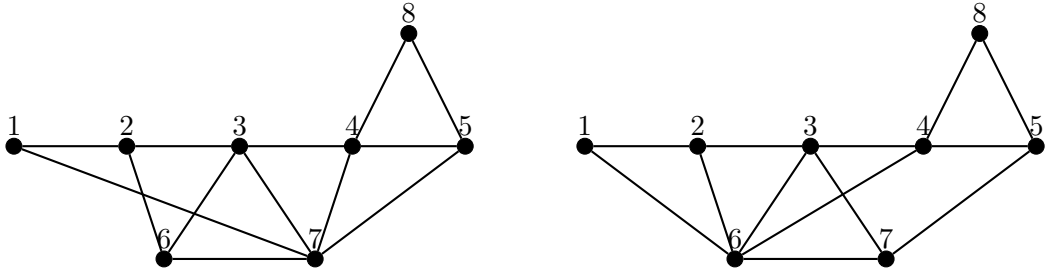
\begin{figure}[!htb]
\begin{center}
\begin{tikzpicture}[scale=1.5]

  \node[label=above:{1}, fill=black, circle] at (0, 1)(1){};
  \node[label=above:{2}, fill=black, circle] at (1, 1)(2){};
  \node[label=above:{3}, fill=black, circle] at (2, 1)(3){};
  \node[label=above:{4}, fill=black, circle] at (3, 1)(4){};
  \node[label=above:{5}, fill=black, circle] at (4, 1)(5){};
  \node[label=above:{6}, fill=black, circle] at (1.33, 0)(6){};
  \node[label=above:{7}, fill=black, circle] at (2.67, 0)(7){};
  \node[label=above:{8}, fill=black, circle] at (3.5, 2)(8){};

  \draw[black, thick] (1) -- (2);
  \draw[black, thick] (2) -- (3);
  \draw[black, thick] (3) -- (4);
  \draw[black, thick] (4) -- (5);
  \draw[black, thick] (4) -- (8);
  \draw[black, thick] (5) -- (8);
  \draw[black, thick] (6) -- (2);
  \draw[black, thick] (6) -- (3);
  \draw[black, thick] (7) -- (1);
  \draw[black, thick] (7) -- (3);
  \draw[black, thick] (7) -- (4);
  \draw[black, thick] (7) -- (5);
  \draw[black, thick] (7) -- (6);

\end{tikzpicture}
\hspace{1cm}
\begin{tikzpicture}[scale=1.5]

  \node[label=above:{1}, fill=black, circle] at (0, 1)(1){};
  \node[label=above:{2}, fill=black, circle] at (1, 1)(2){};
  \node[label=above:{3}, fill=black, circle] at (2, 1)(3){};
  \node[label=above:{4}, fill=black, circle] at (3, 1)(4){};
  \node[label=above:{5}, fill=black, circle] at (4, 1)(5){};
  \node[label=above:{6}, fill=black, circle] at (1.33, 0)(6){};
  \node[label=above:{7}, fill=black, circle] at (2.67, 0)(7){};
  \node[label=above:{8}, fill=black, circle] at (3.5, 2)(8){};

  \draw[black, thick] (1) -- (2);
  \draw[black, thick] (2) -- (3);
  \draw[black, thick] (3) -- (4);
  \draw[black, thick] (4) -- (5);
  \draw[black, thick] (4) -- (8);
  \draw[black, thick] (5) -- (8);
  \draw[black, thick] (6) -- (1);
  \draw[black, thick] (6) -- (2);
  \draw[black, thick] (6) -- (3);
  \draw[black, thick] (6) -- (4);
  \draw[black, thick] (7) -- (3);
  \draw[black, thick] (7) -- (5);
  \draw[black, thick] (7) -- (6);

\end{tikzpicture}
\end{center}
\caption{Graphs $G_1$ and $G_2$ with equal $XB$.}
\label{fig:graphEqual}
\end{figure}

Let the graph on the left be called $G_1$, and the graph on the right $G_2$.  First, note the graphs are indeed nonisomorphic, since for example $G_1$ has the vertex $1$ with degree two that is not part of a triangle, but in $G_2$ both vertices of degree two are in triangles.

To show that these graphs have the same Tutte symmetric function, we apply Lemma \ref{lem:eqx} to the edge $(6,7)$ in both graphs, and reduce to showing that the edge-deleted graphs are $w$-isomorphic, and the edge-contracted graphs are $w$-isomorphic.  It is easy to verify that both edge-contracted graphs are isomorphic to the graph in Figure \ref{fig:graphCon}.

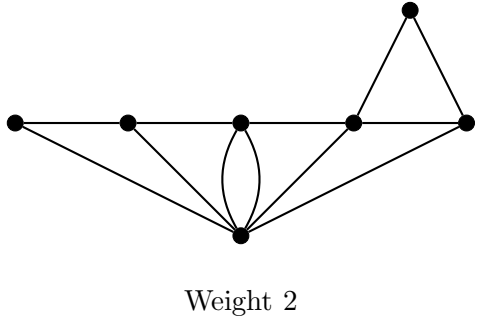
\begin{figure}[!htb]
\begin{center}
\begin{tikzpicture}[scale=1.5]
  \node[fill=black, circle] at (0, 1)(1){};
  \node[fill=black, circle] at (1, 1)(2){};
  \node[fill=black, circle] at (2, 1)(3){};
  \node[fill=black, circle] at (3, 1)(4){};
  \node[fill=black, circle] at (4, 1)(5){};
  \node[label=below:{Weight 2}, fill=black, circle] at (2, 0)(6){};
  \node[fill=black, circle] at (3.5, 2)(8){};

  \draw[black, thick] (1) -- (2);
  \draw[black, thick] (2) -- (3);
  \draw[black, thick] (3) -- (4);
  \draw[black, thick] (4) -- (5);
  \draw[black, thick] (4) -- (8);
  \draw[black, thick] (5) -- (8);
  \draw[black, thick] (6) -- (1);
  \draw[black, thick] (6) -- (2);
  \draw[black, thick] (6) edge [bend left=30] (3);
  \draw[black, thick] (6) -- (4);
  \draw[black, thick] (6) edge [bend right=30] (3);
  \draw[black, thick] (6) -- (5);
\end{tikzpicture}
\end{center}
\caption{The graph formed by contracting $(6,7)$ in $G_1$ or $G_2$.}
\label{fig:graphCon}
\end{figure}

To see that the edge-deleted graphs are isomorphic, take the first graph, delete the edge $(6,7)$ and then rearrange the vertices as in Figure \ref{fig:graphDel}.

\begin{figure}[!htb]
\begin{center}
  \begin{tikzpicture}[scale=1.5]

  \node[label=above:{1}, fill=black, circle] at (2.67, 0)(1){};
  \node[label=above:{2}, fill=black, circle] at (4, 1)(2){};
  \node[label=above:{3}, fill=black, circle] at (3, 1)(3){};
  \node[label=above:{4}, fill=black, circle] at (1.33, 0)(4){};
  \node[label=above:{5}, fill=black, circle] at (1, 1)(5){};
  \node[label=above:{6}, fill=black, circle] at (3.5, 2)(6){};
  \node[label=above:{7}, fill=black, circle] at (2, 1)(7){};
  \node[label=above:{8}, fill=black, circle] at (0, 1)(8){};

  \draw[black, thick] (1) -- (2);
  \draw[black, thick] (2) -- (3);
  \draw[black, thick] (3) -- (4);
  \draw[black, thick] (4) -- (5);
  \draw[black, thick] (4) -- (8);
  \draw[black, thick] (5) -- (8);
  \draw[black, thick] (6) -- (2);
  \draw[black, thick] (6) -- (3);
  \draw[black, thick] (7) -- (1);
  \draw[black, thick] (7) -- (4);
  \draw[black, thick] (7) -- (3);
  \draw[black, thick] (7) -- (5);

  \end{tikzpicture}
\end{center}
\caption{$G_1$ with the edge $(6,7)$ deleted.}
\label{fig:graphDel}
\end{figure}
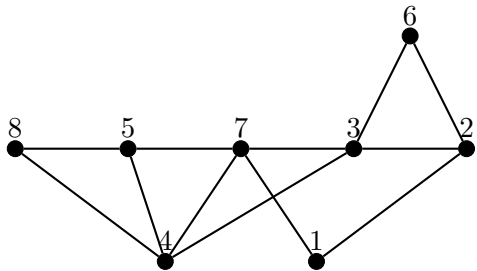

Using this figure, it is easy to see that $G_1 \bk (6,7)$ is $w$-isomorphic to $G_2 \bk (6,7)$, and this shows that the graphs $G_1$ and $G_2$ have equal Tutte symmetric functions.

For a second example, we consider the graphs in Figure \ref{fig:graph2} (with the edges highlighted in red that we will apply Lemma \ref{lem:eqx}).  Let the graph on the top of this figure be $H_1$, and the graph on the bottom be $H_2$.

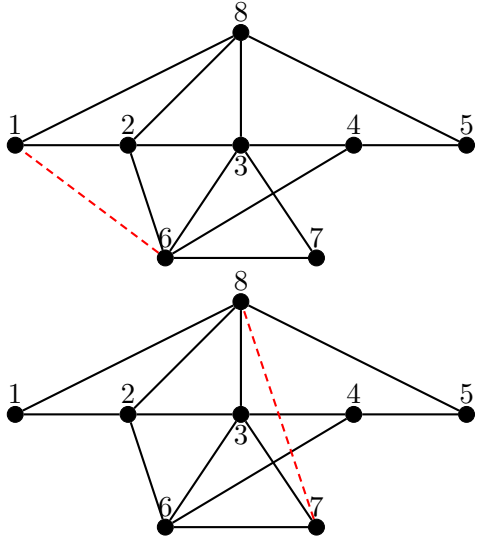
\begin{figure}[!htb]
\begin{center}
  \begin{tikzpicture}[scale=1.5]

  \node[label=above:{1}, fill=black, circle] at (0, 1)(1){};
  \node[label=above:{2}, fill=black, circle] at (1, 1)(2){};
  \node[label=below:{3}, fill=black, circle] at (2, 1)(3){};
  \node[label=above:{4}, fill=black, circle] at (3, 1)(4){};
  \node[label=above:{5}, fill=black, circle] at (4, 1)(5){};
  \node[label=above:{6}, fill=black, circle] at (1.33, 0)(6){};
  \node[label=above:{7}, fill=black, circle] at (2.67, 0)(7){};
  \node[label=above:{8}, fill=black, circle] at (2, 2)(8){};

  \draw[black, thick] (1) -- (2);
  \draw[black, thick] (2) -- (3);
  \draw[black, thick] (3) -- (4);
  \draw[black, thick] (4) -- (5);
  \draw[black, thick] (8) -- (2);
  \draw[black, thick] (8) -- (3);
  \draw[black, thick] (8) -- (1);
  \draw[black, thick] (8) -- (5);
  \draw[black, thick] (6) -- (2);
  \draw[black, thick] (6) -- (3);
  \draw[red, densely dashed, thick] (6) -- (1);
  \draw[black, thick] (6) -- (4);
  \draw[black, thick] (7) -- (6);
  \draw[black, thick] (7) -- (3);

  \end{tikzpicture}

  \begin{tikzpicture}[scale=1.5]

  \node[label=above:{1}, fill=black, circle] at (0, 1)(1){};
  \node[label=above:{2}, fill=black, circle] at (1, 1)(2){};
  \node[label=below:{3}, fill=black, circle] at (2, 1)(3){};
  \node[label=above:{4}, fill=black, circle] at (3, 1)(4){};
  \node[label=above:{5}, fill=black, circle] at (4, 1)(5){};
  \node[label=above:{6}, fill=black, circle] at (1.33, 0)(6){};
  \node[label=above:{7}, fill=black, circle] at (2.67, 0)(7){};
  \node[label=above:{8}, fill=black, circle] at (2, 2)(8){};

  \draw[black, thick] (1) -- (2);
  \draw[black, thick] (2) -- (3);
  \draw[black, thick] (3) -- (4);
  \draw[black, thick] (4) -- (5);
  \draw[black, thick] (8) -- (2);
  \draw[black, thick] (8) -- (3);
  \draw[black, thick] (8) -- (1);
  \draw[black, thick] (8) -- (5);
  \draw[black, thick] (6) -- (2);
  \draw[black, thick] (6) -- (3);
  \draw[black, thick] (6) -- (4);
  \draw[black, thick] (7) -- (6);
  \draw[black, thick] (7) -- (3);
  \draw[red, densely dashed, thick] (7) -- (8);

  \end{tikzpicture}
\end{center}
\caption{Graphs $H_1$ and $H_2$ with nonequal $XB$.}
\label{fig:graph2}
\end{figure}

 First, we verify that $H_1$ and $H_2$ are nonisomorphic.  Both graphs have exactly two vertices of degree $2$, namely vertices $5$ and $7$ in $H_1$ and vertices $5$ and $1$ in $H_2$.  However, in $H_2$ these two vertices have a common neighbor, as they are both adjacent to $8$, but the corresponding vertices do not have a common neighbor in $H_1$.

 To show that $H_1$ and $H_2$ nonetheless have equal Tutte symmetric function, we apply Lemma \ref{lem:eqx} to the edges marked in red in Figure \ref{fig:graph2}.  Clearly the graphs with these edges deleted are isomorphic.  The contracted graphs are shown in Figure \ref{fig:graph2d}.

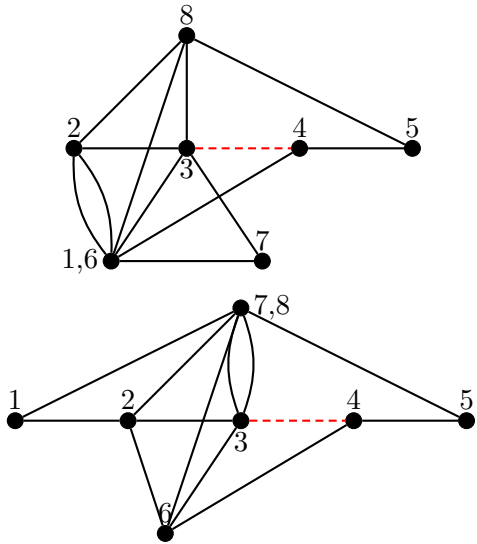
\begin{figure}[!htb]
\begin{center}
  \begin{tikzpicture}[scale=1.5]

  \node[label=above:{2}, fill=black, circle] at (1, 1)(2){};
  \node[label=below:{3}, fill=black, circle] at (2, 1)(3){};
  \node[label=above:{4}, fill=black, circle] at (3, 1)(4){};
  \node[label=above:{5}, fill=black, circle] at (4, 1)(5){};
  \node[label=left:{1,6}, fill=black, circle] at (1.33, 0)(6){};
  \node[label=above:{7}, fill=black, circle] at (2.67, 0)(7){};
  \node[label=above:{8}, fill=black, circle] at (2, 2)(8){};

  \draw[black, thick] (6) edge [bend left=20] (2);
  \draw[black, thick] (2) -- (3);
  \draw[red, densely dashed, thick] (3) -- (4);
  \draw[black, thick] (4) -- (5);
  \draw[black, thick] (8) -- (2);
  \draw[black, thick] (8) -- (3);
  \draw[black, thick] (8) -- (6);
  \draw[black, thick] (8) -- (5);
  \draw[black, thick] (6) edge [bend right=20] (2);
  \draw[black, thick] (6) -- (3);
  \draw[black, thick] (6) -- (4);
  \draw[black, thick] (7) -- (6);
  \draw[black, thick] (7) -- (3);

  \end{tikzpicture}

  \begin{tikzpicture}[scale=1.5]

  \node[label=above:{1}, fill=black, circle] at (0, 1)(1){};
  \node[label=above:{2}, fill=black, circle] at (1, 1)(2){};
  \node[label=below:{3}, fill=black, circle] at (2, 1)(3){};
  \node[label=above:{4}, fill=black, circle] at (3, 1)(4){};
  \node[label=above:{5}, fill=black, circle] at (4, 1)(5){};
  \node[label=above:{6}, fill=black, circle] at (1.33, 0)(6){};
  \node[label=right:{7,8}, fill=black, circle] at (2, 2)(8){};

  \draw[black, thick] (1) -- (2);
  \draw[black, thick] (2) -- (3);
  \draw[red, densely dashed, thick] (3) -- (4);
  \draw[black, thick] (4) -- (5);
  \draw[black, thick] (8) -- (2);
  \draw[black, thick] (8) edge [bend left=20] (3);
  \draw[black, thick] (8) -- (1);
  \draw[black, thick] (8) -- (5);
  \draw[black, thick] (6) -- (2);
  \draw[black, thick] (6) -- (3);
  \draw[black, thick] (6) -- (4);
  \draw[black, thick] (8) -- (6);
  \draw[black, thick] (8) edge [bend right=20] (3);

  \end{tikzpicture}
\end{center}
\caption{$H_1$ and $H_2$ after contraction.}
\label{fig:graph2d}
\end{figure}

The graphs in Figure \ref{fig:graph2d} are not $w$-isomorphic, but it suffices to show that they have equal $XB$, which also provides an example of a pair of non-trivially weighted graphs with equal Tutte symmetric function.  To do so, we again apply Lemma \ref{lem:eqx}.  First, in this figure, we rearrange the top graph $H_1 / (1,6)$ into the graph in Figure \ref{fig:graph3d}.

\begin{figure}[!htb]
\begin{center}
  \begin{tikzpicture}[scale=1.5]

\node[label=above:{7}, fill=black, circle] at (0, 1)(1){};
  \node[label=above:{3}, fill=black, circle] at (1, 1)(2){};
  \node[label=above:{2}, fill=black, circle] at (2, 1)(3){};
  \node[label=above:{5}, fill=black, circle] at (3, 1)(4){};
  \node[label=above:{4}, fill=black, circle] at (4, 1)(5){};
  \node[label=above:{8}, fill=black, circle] at (1.33, 0)(6){};
  \node[label=above:{1,6}, fill=black, circle] at (2, 2)(8){};

  \draw[black, thick] (1) -- (2);
  \draw[red, densely dashed, thick] (2) edge [bend right=20] (5);
  \draw[black, thick] (2) -- (3);
  \draw[black, thick] (4) -- (5);
  \draw[black, thick] (8) -- (2);
  \draw[black, thick] (8) edge [bend left=20] (3);
  \draw[black, thick] (8) -- (1);
  \draw[black, thick] (8) -- (5);
  \draw[black, thick] (6) -- (2);
  \draw[black, thick] (6) -- (3);
  \draw[black, thick] (6) -- (4);
  \draw[black, thick] (8) -- (6);
  \draw[black, thick] (8) edge [bend right=20] (3);

  \end{tikzpicture}
\end{center}
\caption{$H_1 / (1,6)$ rearranged.}
\label{fig:graph3d}
\end{figure}
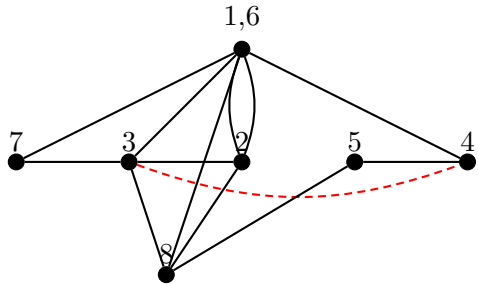

Clearly now the graphs in Figure \ref{fig:graph2d} are $w$-isomorphic with the red edges deleted.  We show that they are also $w$-isomorphic with the red edges contracted.  Those graphs correspond to $H_1$ and $H_2$ with two edges contracted and are shown in Figure \ref{fig:graph4d}.

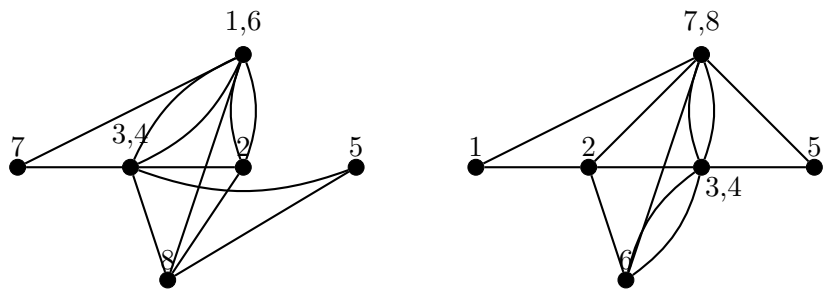
\begin{figure}[!htb]
\begin{center}
  \begin{tikzpicture}[scale=1.5]

  \node[label=above:{7}, fill=black, circle] at (0, 1)(1){};
  \node[label=above:{3,4}, fill=black, circle] at (1, 1)(2){};
  \node[label=above:{2}, fill=black, circle] at (2, 1)(3){};
  \node[label=above:{5}, fill=black, circle] at (3, 1)(4){};
  \node[label=above:{8}, fill=black, circle] at (1.33, 0)(6){};
  \node[label=above:{1,6}, fill=black, circle] at (2, 2)(8){};

  \draw[black, thick] (1) -- (2);
  \draw[black, thick] (2) -- (3);
  \draw[black, thick] (4) edge [bend left = 20] (2);
  \draw[black, thick] (8) edge [bend left = 20] (2);
  \draw[black, thick] (8) edge [bend left = 20] (3);
  \draw[black, thick] (8) -- (1);
  \draw[black, thick] (8) edge [bend right = 20] (2);
  \draw[black, thick] (6) -- (2);
  \draw[black, thick] (6) -- (3);
  \draw[black, thick] (6) -- (4);
  \draw[black, thick] (8) -- (6);
  \draw[black, thick] (8) edge [bend right=20] (3);

  \end{tikzpicture}
  \hspace{1cm}
  \begin{tikzpicture}[scale=1.5]

  \node[label=above:{1}, fill=black, circle] at (0, 1)(1){};
  \node[label=above:{2}, fill=black, circle] at (1, 1)(2){};
  \node[label=below right:{3,4}, fill=black, circle] at (2, 1)(3){};
  \node[label=above:{5}, fill=black, circle] at (3, 1)(5){};
  \node[label=above:{6}, fill=black, circle] at (1.33, 0)(6){};
  \node[label=above:{7,8}, fill=black, circle] at (2, 2)(8){};

  \draw[black, thick] (1) -- (2);
  \draw[black, thick] (2) -- (3);
  \draw[black, thick] (3) -- (5);
  \draw[black, thick] (8) -- (2);
  \draw[black, thick] (8) edge [bend left=20] (3);
  \draw[black, thick] (8) -- (1);
  \draw[black, thick] (8) -- (5);
  \draw[black, thick] (6) -- (2);
  \draw[black, thick] (6) edge [bend left = 20] (3);
  \draw[black, thick] (6) edge [bend right = 20] (3);
  \draw[black, thick] (8) -- (6);
  \draw[black, thick] (8) edge [bend right=20] (3);

  \end{tikzpicture}
\end{center}
\caption{$H_1$ and $H_2$ with two edges contracted.}
\label{fig:graph4d}
\end{figure}

To illustrate that the two graphs in Figure \ref{fig:graph4d} are $w$-isomorphic, we redraw the top graph as shown in Figure \ref{fig:graph5d}, making the $w$-isomorphism apparent.

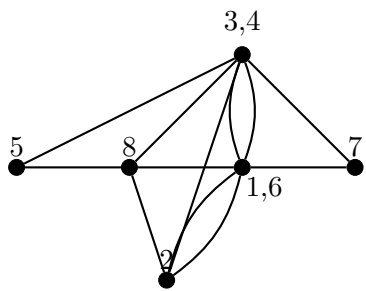
\begin{figure}[!htb]
\begin{center}
  \begin{tikzpicture}[scale=1.5]

  \node[label=above:{5}, fill=black, circle] at (0, 1)(1){};
  \node[label=above:{8}, fill=black, circle] at (1, 1)(2){};
  \node[label=below right:{1,6}, fill=black, circle] at (2, 1)(3){};
  \node[label=above:{7}, fill=black, circle] at (3, 1)(5){};
  \node[label=above:{2}, fill=black, circle] at (1.33, 0)(6){};
  \node[label=above:{3,4}, fill=black, circle] at (2, 2)(8){};

  \draw[black, thick] (1) -- (2);
  \draw[black, thick] (2) -- (3);
  \draw[black, thick] (3) -- (5);
  \draw[black, thick] (8) -- (2);
  \draw[black, thick] (8) edge [bend left=20] (3);
  \draw[black, thick] (8) -- (1);
  \draw[black, thick] (8) -- (5);
  \draw[black, thick] (6) -- (2);
  \draw[black, thick] (6) edge [bend left = 20] (3);
  \draw[black, thick] (6) edge [bend right = 20] (3);
  \draw[black, thick] (8) -- (6);
  \draw[black, thick] (8) edge [bend right=20] (3);

  \end{tikzpicture}
\end{center}
\caption{$H_1$ with two edges contracted rearranged.}
\label{fig:graph5d}
\end{figure}
\end{subsection}

\begin{section}{Further Constructions via the $W$-polynomial}
In \cite{delcon}, the second and third authors noted that vertex-weighted trees are not always distinguished by their chromatic symmetric functions, but the counterexample given in that paper had trees that were isomorphic when unweighted. Other examples of vertex-weighted paths with the same chromatic symmetric function appear in \cite{loebl}. In this section we will develop an algebraic method to construct vertex-weighted graphs with the same $W$-polynomial; by Lemma \ref{lem:Wpoly} this method  also yields vertex-weighted graphs with the same $XB$. The method extends some of the  constructions in \cite{proper} by the first and fourth authors. In an independent work \cite{aliniaeifard2020extended}, Aliniaeifard, Wang and van Willigenburg obtained results similar to the ones presented in this section, but written completely in the language of symmetric functions. 

A \emph{$2$-pointed vertex-weighted graph} is a tuple $(G,w,s,t)$, where $(G,w)$ is a vertex-weighted graph, and $s$ and $t$ are (possibly the same) vertices of $G$. When $w,s,$ and $t$ are clear, we will often just write $G$ in a slight abuse of notation. Given two $2$-pointed vertex-weighted graphs  $(G,w,s,t)$ and $(H,w',s',t')$, define $(G,w,s,t) \cdot (H,w',s',t') = (G\cdot H,w\cdot w',s,t')$, where $G \cdot H$ is the graph formed by taking the disjoint union of $G$ and $H$ and then adding an edge between $t$ and $s'$, and the weight function $w \cdot w'$ assigns to each vertex the weight it had in $G$ or $H$. Second, define $(G,w,s,t) \odot (H,w',s',t')$ as the $2$-pointed vertex-weighted graph $(G \odot H, w \odot w', s, t')$ where $G \odot H = (G \cdot H)/e$ where $e$ is the edge connecting $t$ and $s'$, and $w \odot w' = (w \cdot w')/e$. Finally, if $(G,w,s,t)$ is a $2$-pointed vertex-weighted graph, its \emph{reversal} is the $2$-pointed vertex-weighted graph $(G,w,t,s)$. When simply using $G$ as shorthand to denote $(G,w,s,t)$, we will write $G^*$ to denote its reversal.

An \emph{integer composition} is a sequence $\beta=\beta_1\beta_2\ldots \beta_k$ where each $\beta_i$ is positive. Given a composition $\beta$, its reversal $\beta^*$ is the composition $\beta_k\beta_{k-1}\ldots\beta_1$. Given two compositions $\alpha$ and $\beta$, we say that $\beta\geq \alpha$ if $\beta$ can be obtained by adding consecutive elements of $\alpha$,\emph{e.g.} $4\ 5\geq 1\ 2\ 1\ 3\ 2$. Clearly this forms a partial order on integer compositions.

Given two compositions $\beta = \beta_1 \ldots \beta_k$ and $\beta' = \beta'_1 \ldots \beta'_{k'}$, define $\beta\cdot \beta'$ to be the integer composition $\beta_1\ldots\beta_k\beta'_1\ldots\beta'_{k'}$ and $\beta\odot \beta'$ to be 
the integer composition $\beta_1\ldots\beta_{k-1}(\beta_k+\beta'_1)\beta'_2\ldots\beta'_{k'}$. Given a composition $\beta=\beta_1\beta_2\ldots\beta_l$, we associate to it the 2-pointed vertex-weighted path $P_\beta = (P,w,s,t)$ where $P=v_1v_2\ldots v_l$, $w(v_i)=\beta_i$ for each $i$ and $s=v_1$ and $t=v_l$. It is not difficult to check that $P_{\beta\cdot \beta'} = P_{\beta}\cdot P_{\beta'}$ and $P_{\beta\odot \beta'} = P_\beta\odot P_{\beta'}$.
The $\mathcal{L}$-polynomial of $\beta$, introduced in \cite{proper}, is defined by 
\[\mathcal{L}_\beta(\mathbf{x}) = \sum_{\gamma\geq \beta} \mathbf{x}_{\lambda(\gamma)},\]
where $\mathbf{x}$ is shorthand for the variables $x_1,x_2,\ldots$, $\lambda(\gamma)$ denotes the partition obtained from $\gamma$ after reordering its elements and $\mathbf{x}_\lambda = x_{\lambda_1}x_{\lambda_2}\cdots x_{\lambda_l}$ if $\lambda$ has length $l$.

\begin{lemma}The following statements hold:
\begin{enumerate}
\item 
Given two $2$-pointed vertex-weighted graphs $(G,w,s,t)$ and $(H,w',s',t')$, we have 
\begin{equation}
\label{eq:w2}
    W_{G\cdot H}(\mathbf{x},y) = W_G(\mathbf{x},y)W_{H}(\mathbf{x},y)+W_{G\odot H}(\mathbf{x},y).
\end{equation}
\item For every composition $\beta$ we have 
\[W_{P_\beta}(\mathbf{x},y) = \mathcal{L}_\beta(\mathbf{x})\]
\item Given two compositions $\beta$ and $\beta'$, we have 
\begin{equation}
\label{eq:L}
\mathcal{L}_{\beta\cdot \beta'}=\mathcal{L}_{\beta} \mathcal{L}_{\beta'}+\mathcal{L}_{\beta\odot \beta'}.
\end{equation}
\end{enumerate}
\end{lemma}

\begin{proof}
The first assertion  follows directly from applying deletion-contraction to $G\cdot H$ and the edge $e$ connecting $t$ with $s'$ and the definitions of $\cdot$ and $\odot$. 

For the second assertion, we proceed by induction on the length of $\beta$. The base case follows easily, since 
if $\beta=\beta_1$, then $\mathcal{L}_\beta(\mathbf{x})=x_{\beta_1}$, which is the same as the $W$-polynomial of an isolated vertex with weight $\beta_1$. For the inductive step, we suppose that the assertion holds for compositions of length $n$ and let $\beta=\beta'\beta_{n+1}$ be any integer composition of length $n+1$, where $\beta'$ is a composition of length $n$. It is easy to check  that $\beta'\odot\beta_{n+1}\geq \beta'$ and that if $\gamma\geq\beta$ but $\gamma\not\geq \beta'\odot\beta_{n+1}$, then $\gamma=\gamma'\beta_{n+1}$ where $\gamma'\geq\beta'$. Thus
\[\mathcal{L}_\beta(\mathbf{x})=\sum_{\gamma\geq\beta}x_{\lambda(\gamma)}=\sum_{ \gamma\geq \beta'\odot\beta_{n+1}}x_{\lambda(\gamma)}+\sum_{\gamma'\geq \beta'}x_{\lambda(\gamma'\beta_{n+1})}=\mathcal{L}_{\beta'\odot\beta_{n+1}}(\mathbf{x})+\mathcal{L}_{\beta'}(\mathbf{x})x_{\beta_{n+1}}.\] 
Hence, by the induction hypothesis, we get from the last equality
\[\mathcal{L}_\beta(\mathbf{x}) = W_{P_{\beta'\odot\beta_{n+1}}}(\mathbf{x},y)+W_{P_{\beta'}}(\mathbf{x},y)W_{P_{\beta_{n+1}}}(\mathbf{x},y)=W_{P_\beta}(\mathbf{x},y),\]
where the last equality follows from \eqref{eq:w2}. This finishes the proof of the second assertion. The third assertion is a direct consequence of the previous two assertions. 
\end{proof}
Next, we will see how to use this lemma to construct arbitrarily large families of vertex-weighted trees with the same $W$-polynomial by gluing together several copies of a seed graph. For a 2-pointed vertex-weighted graph $G$, we denote the graph $G\cdot G\cdot \ldots \cdot G$ by $G^i$ and the graph $G\odot G\odot\ldots\odot G$ by $G^{\odot i}$ (in each definition the graph $G$ is repeated $i$ times). 

If $G$ a 2-pointed vertex-weighted graph and $\beta$ is an integer composition, then define
\[ \beta\circ G := G^{\odot \beta_1}\cdot G^{\odot\beta_2}\cdots G^{\odot\beta_k}.\]
Similarly, if $\gamma$ is another integer composition, we can define 
\[ \beta\circ \gamma:= \gamma^{\odot \beta_1}\cdot \gamma^{\odot\beta_2}\cdots \gamma^{\odot\beta_k}.\]
It is not difficult to check that $\beta\circ P_\gamma = P_{\beta\circ \gamma}$ and that 
\begin{equation}
\label{eq:distribute}
(\alpha *\beta)\circ\gamma = (\alpha\circ \gamma)*(\beta\circ\gamma),
\end{equation}
 where $\alpha,\beta$ are compositions, $\gamma$ can be either a $2$-pointed vertex-weighted graph or a composition, and $*$ may be either $\cdot$ or $\odot$. 
\begin{lemma}
The operation $\circ$ in the set of integer compositions is associative. Moreover, if $(G,w,s,t)$ is a $2$-pointed vertex-weighted graph and $\alpha$ and $\beta$ are integer compositions, then 
\[ (\alpha\circ\beta)\circ G = \alpha\circ(\beta\circ G).\] 
\end{lemma}
\begin{proof}
The associativity of $\circ$ for integer compositions is shown in \cite[Proposition 3.3]{decomposable}. The second assertion will be shown by induction on the length of $\alpha$. If $\alpha=\alpha_1$, then 
\begin{align*}
\alpha_1\circ(\beta\circ G) = (\beta\circ G)^{\odot\alpha_1} &= 
(G^{\odot \beta_1}\cdots G^{\odot \beta_k})^{\odot \alpha_1}\\
&= (G^{\odot \beta_1}\cdots G^{\odot \beta_k})\odot (G^{\odot \beta_1}\cdots G^{\odot \beta_k})\odot\cdots\odot
(G^{\odot \beta_1}\cdots G^{\odot \beta_k})\\
&=G^{\odot \beta_1}\cdots G^{\odot \beta_k}\odot G^{\odot \beta_1}\cdots G^{\odot \beta_k}\odot\cdots\odot
G^{\odot \beta_1}\cdots G^{\odot \beta_k}\\
&=G^{\odot \beta_1}\cdots G^{\odot \beta_{k-1}}\cdot (G^{\odot \beta_k+\beta_1}\cdot G^{\odot \beta_2}\cdots G^{\odot \beta_{k-1}})^{\alpha_1-1}G^{\odot \beta_k}\\
&= \Big(\beta_1\cdots\beta_{k-1}\big((\beta_{k}+\beta_1)\beta_2\cdots\beta_{k-1}\big)^{\alpha_1-1}\beta_k\Big)\circ G = (\alpha_1\circ\beta)\circ G,
\end{align*}
which shows the base case. For the inductive step, given any composition $\alpha$ of length $l$, we may write it as $\alpha = \alpha'\alpha_{l}$ where $\alpha'$ is a composition of length $l-1$. By applying \eqref{eq:distribute}, we get
\[ (\alpha'\alpha_l)\circ (\beta\circ G) = \big(\alpha'\circ (\beta\circ G)\big)\cdot \big(\alpha_l\circ (\beta\circ G)\big).\]
By applying the induction hypothesis twice and then \eqref{eq:distribute} again, the expresion in the right hand side of the last equation 
becomes
\[((\alpha'\circ\beta)\circ G) \cdot ((\alpha_l\circ\beta)\circ G)
 =\big((\alpha'\circ\beta\big) \cdot (\alpha_l\circ\beta))\circ G = (\alpha\circ \beta)\circ G.\]
This finishes the induction, and hence the proof.
\end{proof}

The utility of the $\circ$ operation is that the $W$-polynomial of $\beta\circ G$ can be computed in terms 
of the $\mathcal{L}$-polynomial of $\beta$ and the $W$-polynomial of $G$.

\begin{theorem}
\label{thm:plet}
Let $(G,w,s,t)$ be a 2-pointed vertex-weighted graph and $\beta$ a composition. We have
\begin{equation}
    \label{eq:circ}
W_{\beta\circ G}(\mathbf{x},y) = \mathcal{L}_\beta (x_i\mapsto W_{G^{\odot i}}(\mathbf{x},y))
\end{equation}
\end{theorem}
\begin{proof}
We do the proof by induction on the length $l$ of $\beta$. For the base case, assume $l=1$ and check that  $\beta\circ G = G^{\odot \beta_1}$ and $\mathcal{L}_\beta(x) = x_{\beta_1}$. Thus, substituting $x_{\beta_1}\mapsto W_{G^{\odot \beta_1}}(\mathbf{x},y)$ yields the assertion when $l=1$.

Now, using the inductive hypothesis, we will assume that the assertion holds for all compositions of length $k$ and consider $\beta=\beta'\beta_{k+1}$, where $\beta'=\beta_1\ldots\beta_k$. By the definition of $\circ$, we have 
\[\beta\circ G = (\beta'\circ G)\cdot (\beta_{k+1}\circ G).\]
thus, applying \eqref{eq:w2} yields

\begin{equation}
\label{eq:w3}
W_{\beta\circ G}(\mathbf{x},y) = W_{\beta'\circ G}(\mathbf{x},y)W_{\beta_{k+1}\circ G}(\mathbf{x},y)+W_{(\beta'\circ G)\odot ({\beta_{k+1}\circ G)}}(\mathbf{x},y).
\end{equation}
It is easy to check that $(\beta'\circ G)\odot (\beta_{k+1}\circ G) = (\beta'\odot \beta_{k+1})\circ G$. On the other hand, by the induction hypothesis,
\[W_{\beta'\circ G}(\mathbf{x},y) = \mathcal{L}_{\beta'} (x_i\mapsto W_{G^{\odot i}}(\mathbf{x},y)),\]
\[W_{\beta_{k+1}\circ G}(\mathbf{x},y) = \mathcal{L}_{\beta_{k+1}} (x_i\mapsto W_{G^{\odot i}}(\mathbf{x},y))\]
and
\[W_{(\beta'\odot\beta_{k+1})\circ G}(\mathbf{x},y) = \mathcal{L}_{(\beta'\odot\beta_{k+1})} (x_i\mapsto W_{G^{\odot i}}(\mathbf{x},y)).\]
Combining these into \eqref{eq:w3}, we get 
\begin{equation}
\label{eq:w4}
W_{\beta\circ G}(\mathbf{x},y) = (\mathcal{L}_{\beta'}\mathcal{L}_{\beta_{k+1}} + \mathcal{L}_{\beta'\odot\beta_{k+1}} )(x_i\mapsto W_{G^{\odot i}}(\mathbf{x},y))\end{equation}
and thence the assertion follows by applying \eqref{eq:L} to the last equation.
\end{proof}
\begin{cor}
\label{cor:weighted}
Let $\beta$ be an integer composition and $(G,w,s,t)$ be a  $2$-pointed vertex-weighted graph. Then for every composition $\gamma$ such that $\mathcal{L}_\gamma=\mathcal{L}_\beta$ we have
\[ W_{\beta\circ G}(\mathbf{x},y) = W_{\gamma\circ G}(\mathbf{x},y).\]
\end{cor}
This motivates us to characterize the class of compositions with the same $\mathcal{L}$-polynomial as a given composition. This characterization was obtained in \cite{decomposable} in the language of ribbon Schur functions and later recast in \cite{proper} to the language we are using here. If a composition $\beta$ is written in the form $\beta_1\circ\beta_2\circ\cdots\beta_k$, then we call this a \emph{factorization} of $\beta$. We say that the factorization $\beta=\beta_1\circ\beta_2$ is
\emph{trivial} if any of the following conditions hold: 
\begin{enumerate}
\item  one of $\beta_1,\beta_2$ are equal to $1$,
\item  the lengths of $\beta_1$ and $\beta_2$ are both equal to $1$,
\item  the compositions $\beta_1$ and $\beta_2$ both have all parts equal to $1$. 
\end{enumerate}

A factorization $\beta = \beta_1\circ\cdots\circ\beta_k$ is \emph{irreducible} if no $\beta_i\circ \beta_{i+1}$ is a trivial factorization, and each $\beta_i$ admits only trivial factorizations. In this case, each $\beta_i$ is called an \emph{irreducible factor}.

\begin{theorem}[\cite{decomposable,proper}]
\label{thm:decomposable} 
The irreducible factorization of any integer composition is unique. 
Moreover  if $\beta = \beta_1\circ\beta_2\circ\cdots\circ\beta_l$ and $\gamma = \gamma_1\circ\gamma_2\circ\cdots\circ\gamma_k$ are two compositions with given irreducible factorizations, then $\beta$ and $\gamma$ have the same $\mathcal{L}$-polynomial if and only if 
\[l = k \quad\text{and}\quad \gamma_i\in\{\beta_i,\beta_i^*\}\text{ for all $i$ in $\{1,\ldots,k\}$}.\] 
\end{theorem}
\begin{remark}
By the second assertion of Lemma \ref{lem:Wpoly} and \eqref{eq:XBsub} and the fact that $X$ can be recovered from $XB$ by setting $t=-1$ (which is equivalent to setting $y=0$ in the $W$-polynomial) we check that substituting each variable $x_i$ by $-p_i$  in $\mathcal{L}_\beta$ yields the weighted chromatic symmetric function of $P_\beta$. On the other hand, if we substitute each variable $x_i$ by $-h_i$, where $h_i$ is the $i$-th complete homogeneous symmetric function, we obtain the ribbon Schur function associated with $\beta$. It follows that the homomorphism $U$ of $\Lambda$ defined by sending each $p_\lambda$ to $h_\lambda$ and then extending linearly transforms the chromatic symmetric function of the weighted path $P_\beta$ into the the ribbon Schur function associated with $\beta$. This observation is implicit in \cite{proper} and the morphism $U$ is studied in detail in \cite{aliniaeifard2020extended}.
\end{remark}

Combining Corollary \ref{cor:weighted} and Theorem \ref{thm:decomposable} we get
\begin{cor}
Let $(G,w,s,t)$ be a 2-pointed weighted graph and $\beta$ an integer composition. Suppose  we have an irreducible  factorization $\beta=\beta_1\circ\beta_2\circ\cdots\circ\beta_k$. Then, all vertex-weighted graphs in the set 
\[ Sym(\beta\circ G)=\{ \gamma_1\circ\gamma_2\circ\cdots\circ\gamma_i\circ G\mid \text{for each $i$}, \gamma_i\in\{\beta_i,\beta_i^*\}\}
\]
have the same $W$-polynomial as $\beta\circ G$.
\end{cor}

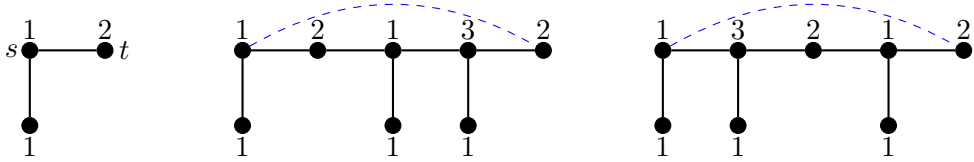
\begin{figure}[!htb]
\begin{center}
\begin{tikzpicture}
\node[label=below:{$1$},, fill=black, circle] at (0, -1)(v1){};
\node[label=above:{$1$},label=left:{$s$}, fill=black, circle] at (0, 0)(v2){};
\node[label=above:{$2$},label=right:{$t$}, fill=black, circle] at (1, 0)(v3){};
\draw[black, thick] (v1) -- (v2);
\draw[black, thick] (v2) -- (v3);
\end{tikzpicture}
\hspace{1cm}
\begin{tikzpicture}
\node[label=below:{$1$}, fill=black, circle] at (0, -1)(v1){};
\node[label=above:{$1$}, fill=black, circle] at (0, 0)(v2){};
\node[label=above:{$2$}, fill=black, circle] at (1, 0)(v3){};
\node[label=below:{$1$}, fill=black, circle] at (2, -1)(v4){};
\node[label=above:{$1$}, fill=black, circle] at (2, 0)(v5){};
\node[label=above:{$3$}, fill=black, circle] at (3, 0)(v6){};
\node[label=below:{$1$}, fill=black, circle] at (3, -1)(v7){};
\node[label=above:{$2$}, fill=black, circle] at (4, 0)(v8){};
\draw[black, thick] (v1) -- (v2);
\draw[black, thick] (v2) -- (v3) -- (v5)--(v6)--(v8);
\draw[black, thick] (v6) -- (v7);
\draw[black, thick] (v5) -- (v4);
\draw[blue, dashed] (v2) edge [bend left = 30] (v8);
\end{tikzpicture} 
\hspace{1cm}
\begin{tikzpicture}
\node[label=above:{$1$}, fill=black, circle] at (0, 0)(v2){};
\node[label=above:{$3$}, fill=black, circle] at (1, 0)(v3){};
\node[label=above:{$2$}, fill=black, circle] at (2, 0)(v5){};
\node[label=above:{$1$}, fill=black, circle] at (3, 0)(v6){};
\node[label=above:{$2$}, fill=black, circle] at (4, 0)(v8){};
\node[label=below:{$1$}, fill=black, circle] at (0, -1)(v1){};
\node[label=below:{$1$}, fill=black, circle] at (1, -1)(v4){};
\node[label=below:{$1$}, fill=black, circle] at (3, -1)(v7){};
\draw[black, thick] (v1) -- (v2);
\draw[black, thick] (v2) -- (v3) -- (v5)--(v6)--(v8);
\draw[black, thick] (v6) -- (v7);
\draw[black, thick] (v3) -- (v4);
\draw[blue, dashed] (v2) edge [bend left = 30] (v8);
\end{tikzpicture} 
\end{center}
\caption{The seed tree $T$, and the trees $12\circ T$ and $21\circ T$. The latter two trees have the same $W$-polynomial (the dashed line is a non-edge).}
\label{fig:123}
\end{figure}

Finally, we apply these results to give two examples of weighted trees with the same $W$-polynomial (and hence Tutte symmetric function). For the first example we consider the $2$-pointed vertex-weighted tree $T$ depicted in Figure \ref{fig:123} 
and the composition $\beta= 1 2$. Then, it follows that the weighted trees $\beta\circ T$ and $\beta^*\circ T$ have the same $W$-polynomial but they are not $w$-isomorphic. In fact, they are not even isomorphic as unweighted trees. These trees are also shown in Figure \ref{fig:123}. However, in this case,  there is an alternate way of checking that the trees have the same $W$-polynomial: It may be checked that the graphs are $w$-isomorphic when adding the dashed non-edge and when contracting the dashed non-edge, so it follows from Lemma \ref{lem:eqx}. Our construction can also be used to obtain a more complicated pair of examples (that cannot be constructed by the direct use of deletion-contraction operations). Consider $\beta\circ T$ and $\gamma\circ T$ where $T$ is the same $2$-pointed vertex-weighted tree depicted in Figure \ref{fig:123}, $\beta=1 2 1 3 2 = 1 2 \circ 1 2$ and $\gamma=1 3 2 1 2 = 2 1 \circ 1 2$. By applying Corollary \ref{cor:weighted} we see that these weighted trees have the same $W$-polynomial, but clearly they are not $w$-isomorphic; they are depicted in Figure \ref{fig:1232}.
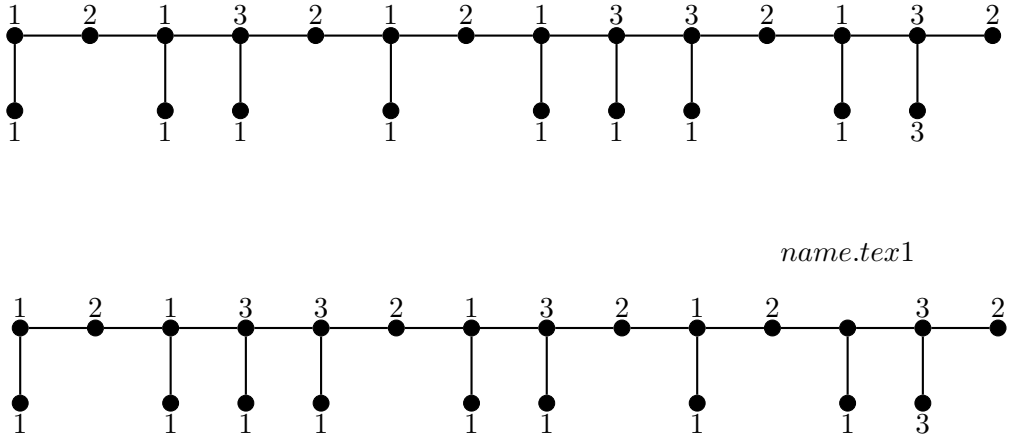
\begin{figure}
\begin{center}
\begin{tikzpicture}
\node[label=above:{$1$}, fill=black, circle] at (0, 0)(v1){};
\node[label=above:{$2$}, fill=black, circle] at (1, 0)(v2){};
\node[label=above:{$1$}, fill=black, circle] at (2, 0)(v3){};
\node[label=above:{$3$}, fill=black, circle] at (3, 0)(v4){};
\node[label=above:{$2$}, fill=black, circle] at (4, 0)(v5){};
\node[label=above:{$1$}, fill=black, circle] at (5, 0)(v6){};
\node[label=above:{$2$}, fill=black, circle] at (6, 0)(v7){};
\node[label=above:{$1$}, fill=black, circle] at (7, 0)(v8){};
\node[label=above:{$3$}, fill=black, circle] at (8, 0)(v9){};
\node[label=above:{$3$}, fill=black, circle] at (9, 0)(v10){};
\node[label=above:{$2$}, fill=black, circle] at (10, 0)(v11){};
\node[label=above:{$1$}, fill=black, circle] at (11, 0)(v12){};
\node[label=above:{$3$}, fill=black, circle] at (12, 0)(v13){};
\node[label=above:{$2$}, fill=black, circle] at (13, 0)(v14){};
\draw[black, thick] (v1) -- (v2) -- (v3) -- (v4) -- (v5)--(v6)--(v7)--(v8)--(v9)--(v10)--(v11)--(v12)--(v13)--(v14);
\node[label=below:{$1$}, fill=black, circle] at (0,-1)(a1) {};
\node[label=below:{$1$}, fill=black, circle] at (2,-1)(a3) {};
\node[label=below:{$1$}, fill=black, circle] at (3,-1)(a4) {};
\node[label=below:{$1$}, fill=black, circle] at (5,-1)(a6) {};
\node[label=below:{$1$}, fill=black, circle] at (7,-1)(a8) {};
\node[label=below:{$1$}, fill=black, circle] at (8,-1)(a9) {};
\node[label=below:{$1$}, fill=black, circle] at (9,-1)(a10) {};
\node[label=below:{$1$}, fill=black, circle] at (11, -1)(a12){};
\node[label=below:{$3$}, fill=black, circle] at (12, -1)(a13){};
\draw[black,thick] (v1)--(a1);
\draw[black,thick] (v3)--(a3);
\draw[black,thick] (v4)--(a4);
\draw[black,thick] (v6)--(a6);
\draw[black,thick] (v8)--(a8);
\draw[black,thick] (v9)--(a9);
\draw[black,thick] (v10)--(a10);
\draw[black,thick] (v12)--(a12);
\draw[black,thick] (v13)--(a13);
\end{tikzpicture} 
\vspace{0.5cm}

\begin{tikzpicture}
\node[label=above:{$1$}, fill=black, circle] at (0, 0)(v1){};
\node[label=above:{$2$}, fill=black, circle] at (1, 0)(v2){};
\node[label=above:{$1$}, fill=black, circle] at (2, 0)(v3){};
\node[label=above:{$3$}, fill=black, circle] at (3, 0)(v4){};
\node[label=above:{$3$}, fill=black, circle] at (4, 0)(v5){};
\node[label=above:{$2$}, fill=black, circle] at (5, 0)(v6){};
\node[label=above:{$1$}, fill=black, circle] at (6, 0)(v7){};
\node[label=above:{$3$}, fill=black, circle] at (7, 0)(v8){};
\node[label=above:{$2$}, fill=black, circle] at (8, 0)(v9){};
\node[label=above:{$1$}, fill=black, circle] at (9, 0)(v10){};
\node[label=above:{$2$}, fill=black, circle] at (10, 0)(v11){};
\node[label=above:{$name.tex1$}, fill=black, circle] at (11, 0)(v12){};
\node[label=above:{$3$}, fill=black, circle] at (12, 0)(v13){};
\node[label=above:{$2$}, fill=black, circle] at (13, 0)(v14){};
\draw[black, thick] (v1) -- (v2) -- (v3) -- (v4) -- (v5)--(v6)--(v7)--(v8)--(v9)--(v10)--(v11)--(v12)--(v13)--(v14);
\node[label=below:{$1$}, fill=black, circle] at (0,-1)(a1) {};
\node[label=below:{$1$}, fill=black, circle] at (2,-1)(a3) {};
\node[label=below:{$1$}, fill=black, circle] at (3,-1)(a4) {};
\node[label=below:{$1$}, fill=black, circle] at (4,-1)(a5) {};
\node[label=below:{$1$}, fill=black, circle] at (7,-1)(a8) {};
\node[label=below:{$1$}, fill=black, circle] at (6,-1)(a7) {};
\node[label=below:{$1$}, fill=black, circle] at (9,-1)(a10) {};
\node[label=below:{$1$}, fill=black, circle] at (11, -1)(a12){};
\node[label=below:{$3$}, fill=black, circle] at (12, -1)(a13){};
\draw[black,thick] (v1)--(a1);
\draw[black,thick] (v3)--(a3);
\draw[black,thick] (v4)--(a4);
\draw[black,thick] (v5)--(a5);
\draw[black,thick] (v8)--(a8);
\draw[black,thick] (v7)--(a7);
\draw[black,thick] (v10)--(a10);
\draw[black,thick] (v12)--(a12);
\draw[black,thick] (v13)--(a13);
\end{tikzpicture} \end{center}
\caption{Weighted trees with the same Tutte symmetric function.}
\label{fig:1232}
\end{figure}

\end{section}

\end{section}

\begin{section}{Further Research}
We conclude with some data and further possible considerations for research.

Using deletion-contraction relations, we computed $X_G$ and $XB_G$ for simple graphs with at most $8$ vertices using data provided by \cite{mckay}.  We found many pairs of such graphs with equal chromatic symmetric function, and for all of these pairs we also determined whether the graphs are distinguished by $XB$.  This information and more may be viewed at \cite{data}.

 We find that triangles seem to play an important role in graphs with equal chromatic symmetric function.  Note that in the $1000$ pairs of graphs with equal chromatic symmetric function noted in \cite{data}, every graph contains a triangle. Furthermore, each of the three methods given in Sections 5.1 and 5.2 for constructing graphs with equal chromatic symmetric function always produces a pair of graphs containing triangles.  In the case of Lemma \ref{Lem:Split} and Theorem \ref{thm:ore} this is explicit.  In the case of Lemma \ref{lem:eqX}, suppose that we have a graph $G$ satisfying the assumptions of the lemma.  If $N(v_3) = \emptyset$, then $G \cup v_1v_3$ is isomorphic to $G \cup v_2v_3$ since by assumption there is an automorphism of $G \bk v_3$ swapping $v_1$ and $v_2$.  If there is a vertex $x \in N(v_3)$, then by assumption also $x \in N(v_1)$ and $x \in N(v_2)$, so in $G \cup v_1v_3$ there is a triangle with vertices $v_1,v_3,x$ and in $G \cup v_2v_3$ there is a triangle with vertices $v_2,v_3,x$.  Thus, every $G$ satisfying the conditions of Lemma \ref{lem:eqX} either produces two isomorphic graphs, or two graphs with equal chromatic symmetric function that both contain triangles.  Finally, we also note the recent result of Penaguiao \cite{raul} showing that given any two nonisomorphic graphs with equal chromatic symmetric function, one may be transformed into the other by a finite number of applications of the triangular modular relation of Orellana and Scott \cite{ore}.

Indeed, prior to the discovery of the construction given in Section 6.3, the authors considered whether triangle-free graphs may be distinguished by the chromatic symmetric function! It would be useful if it could be demonstrated an explicit sequence of applications of the modular relation of \cite{ore} that takes these triangle-free graphs to each other.  As far as distinguishing graphs goes, the next logical question is to determine whether there are bipartite graphs with equal chromatic symmetric function.

The spanning tree formula \eqref{eq:spanX} for the chromatic symmetric function is new, and may be useful to ongoing research.  Furthermore, the sum runs over those spanning trees of a graph $G$ with no external activity.  It is worth noting that the number of this particular kind of spanning tree in a graph $G$ is equal to the number of $G$-parking functions with respect to any vertex, and the number of acyclic orientations of $G$ with exactly one sink \cite{gparking}.  It would be interesting to see if there are similar expansions to \eqref{eq:spanX} that run over one of these sets.

Additionally, there is an expansion of the Tutte polynomial as a sum over $G$-parking functions given in \cite{tutteG} as 
$$
T_G(x,y) = \sum_{f} x^{cb(f)}y^{w(f)}
$$
where $cb(f)$ is the number of critical bridge vertices of $G$ with respect to $f$ (as defined in \cite{tutteG}), and $w(f) = |E(G)|-|V(G)|-\sum_{v \in V(G)} f(v)$.  In the same way that the spanning tree formula \eqref{eq:tree} for $XB$ is an extension of a similar formula for $T_G$, perhaps there is a natural formula for $XB$ that extends this $G$-parking function expansion of $T_G$.

Finally, there appear to be many rich, unexplored connections between the Tutte symmetric function and other functions derived from the $V$-polynomial.  For example, one can use a specialization of the $V$-polynomial as a natural list-coloring polynomial \cite{listcolor}.  It would be interesting to see if the Tutte symmetric function could be modified to consider this or other $V$-polynomial specializations.

\end{section}

\begin{section}{Acknowledgments}
The authors would like to thank Jo Ellis-Monaghan for telling us about the $V$-polynomial, Martin Loebl for discussions about the split graph construction and the $U$-polynomial, and Steve Noble, Bruce Sagan, Darij Grinberg, Steph van Willigenburg, Farid Aliniaiefard  and Victor Wang for many helpful comments. We would like to thank the anonymous referee for their thoughtful comments and suggestions.

The authors would also like to acknowledge Brendan McKay's webpage of combinatorial data \cite{mckay}.  Its database of graphs and trees has been a valuable resource for creating and testing our conjectures, and was helpful for creating \cite{data}.

The first version of this article did not contain the results of Section 7. While working on this version of the article, the authors became aware of the independent work \cite{aliniaeifard2020extended}. 

This material is based upon work supported by the National Science Foundation under Award No. DMS-1802201. We acknowledge the support of the Natural Sciences and Engineering Research Council of Canada (NSERC), [funding
reference number RGPIN-2020-03912]. J.A.-P. and J.Z. acknowledge support from CONICYT FONDECYT REGULAR 1160975.

\end{section}

\bibliographystyle{plain}
\bibliography{bib}

\end{document}